\documentclass[12pt, letterpaper]{article}
\usepackage{verbatim}
\usepackage{amsfonts}
\usepackage{amsmath}
\usepackage{amssymb}
\usepackage{amsthm}
\usepackage{caption}
\usepackage{cleveref}
\usepackage{enumitem}
\usepackage[margin=1in]{geometry}
\usepackage{graphicx}
\usepackage{stix}
\usepackage{tikz}
\usepackage{tikz-cd}
\usetikzlibrary{arrows,calc,decorations.markings,patterns,shapes}
\usepackage{xcolor}
\usepackage{mathtools}
\newcommand{\Th}{\operatorname{Th}}
\newcommand{\Z}{\mathbb{Z}}
\newcommand{\G}{\Gamma}
\newcommand{\var}{\varphi}

\renewcommand{\a}{\alpha}

\newcommand{\lk}{\operatorname{lk}}

\newcommand{\cdim}{\operatorname{cd}}
\newcommand{\cd}{\operatorname{cd}}
\newcommand{\vcd}{\operatorname{vcd}}

\def\S{\mathcal{S}}
\def\Z{\mathbb{Z}}
\theoremstyle{plain}
\newtheorem*{mainthm}{Main Theorem}
\newtheorem{thm}{Theorem}[section]
\newtheorem{cor}[thm]{Corollary}
\newtheorem{lem}[thm]{Lemma}
\newtheorem{prop}[thm]{Proposition}

\theoremstyle{definition}
\newtheorem{defn}[thm]{Definition}
\newtheorem{rmk}[thm]{Remark}

\definecolor{cbred}{HTML}{A40725}
\definecolor{cbblue}{HTML}{0057B2}
\definecolor{cbgreen}{HTML}{08AD15}
\definecolor{cbviolet}{HTML}{6C016D}

\title{High dimensional hyperbolic Coxeter groups that virtually fiber}
\author{
    Jean-Fran\c{c}ois Lafont\\
    \texttt{jlafont@math.ohio-state.edu} \\
\and 
    Barry Minemyer \\
    \texttt{bminemyer@commonwealthu.edu} \\
\and 
    Gangotryi Sorcar \\
    \texttt{gangotryi.sorcar@krea.edu.in}
\and 
    Matthew Stover \\
    \texttt{mstover@temple.edu}
\and 
    Joseph Wells \\
    \texttt{jwellsmath@protonmail.com}
    }
\date{\today}

\begin{document}

\maketitle

\begin{abstract}
    This paper provides an iterative procedure for constructing hyperbolic Coxeter groups that virtually fiber over $\Z$ that is flexible enough to yield infinitely many isomorphism classes in each virtual cohomological dimension (vcd) $n\geq 2$. Our procedure combines results of Jankiewicz, Norin, and Wise with a generalization of a construction due to Osajda involving a new simplicial thickening process. We also give a topological argument showing that the vcd of the right-angled Coxeter groups produced by our construction increases by exactly one with each iteration, guaranteeing that our process produces examples of every vcd.
\end{abstract}

\section{Introduction}

A group is said to \emph{virtually algebraically fiber}, or just to \emph{virtually fiber}, if it has a finite index subgroup that surjects onto $\Z$ with finitely generated kernel. In a seminal paper, Bestvina and Brady developed a combinatorial Morse Theory approach to determining whether certain groups virtually algebraically fiber \cite{BB}. More recently, Jankiewicz, Norin, and Wise gave a combinatorial criterion, based only on the existence of a certain state on the defining graph of a right-angled Coxeter group that forms a \emph{legal orbit} under a certain set of moves, to certify that Bestvina--Brady Morse Theory applies to provide a virtual fibration \cite{JNW}. This result motivated a large number of recent results in geometric group theory and hyperbolic geometry, including examples of higher-dimensional hyperbolic manifolds whose fundamental groups virtually fiber \cite{BM,IMM1,F} and the first hyperbolic 5-manifold known to topologically fiber over the circle \cite{IMM2}.

The purpose of this paper is to construct new examples of Gromov hyperbolic right-angled Coxeter groups that virtually fiber. Other than examples found in the above references, the only other infinite family of examples we know are due to Schesler and Zaremsky \cite[Thm.\ 10.2]{ScheslerZaremsky}. As Schesler and Zaremsky remark after their proof, their examples all have cohomological dimension two. Our contribution is to construct infinitely many isomorphism classes in every cohomological dimension.

\begin{mainthm}\label{thm:main}
    For every $n \geq 2$, there exist infinitely many isomorphism classes of hyperbolic right-angled Coxeter groups that virtually algebraically fiber and have virtual cohomological dimension $n$.
\end{mainthm}

We now describe the construction. In \cite{osajda}, Osajda developed a recursive construction that builds hyperbolic right-angled Coxeter groups with arbitrarily large virtual cohomological dimension. In general, these Coxeter groups may not satisfy the hypotheses of \cite[Thm.\ 6.14]{JNW}, and therefore the combinatorial systems of moves approach developed in that paper may not apply to provide a virtual algebraic fibration. We combine the recursive construction in \cite{osajda} with a novel ``$\alpha$-thickening'' procedure generalizing the thickening procedure used by Osajda to construct hyperbolic right-angled Coxeter groups whose defining graphs admit a state with a legal orbit under an appropriate system of moves, so \cite{JNW} then applies to provide a virtual fibration.

Given a finite cube complex $X$, our $\alpha$-thickening procedure enlarges the vertex set  by a surjection $\alpha \colon Y \to V(X)$ from a large finite set $Y$, inserts a high-dimensional simplex above each vertex of the original cube complex, and then applies the standard thickening procedure of \cite{osajda}. This has the effect of inserting enough vertices and edges into the 1-skeleton of the complex so that one can develop a legal system of moves without changing the homotopy type of the resulting simplicial complex. This process may be of independent interest and is developed in \Cref{sec:alpha-thickening} below.

We now describe the general strategy for producing our virtually fibered hyperbolic right-angled Coxeter groups in more detail. We work with cube complexes that are $5$-large (see \Cref{sec:alpha-thickening}) and satisfy the conditions of having \emph{no isolated corners} and \emph{no disconnecting cubes}; see \Cref{sec:thickening} for definitions.

\begin{enumerate}
    \item Start with a finite $5$-large cube complex $X_n$ that has no isolated corners, no disconnecting cubes, and has (integral) cohomological dimension $n$. 

    \item Apply a carefully chosen $\alpha$-thickening procedure to $X_n$ to produce the simplicial complex $T_n\coloneqq\Th_\alpha(X_n)$. Since $\Th_\alpha(X_n)$ is homotopy equivalent to $X_n$, the simplicial complex $T_n$ also has cohomological dimension $n$.

    \item Using properties of the $\alpha$-thickening, we verify that the $1$-skeleton of $T_n$ is $5$-large and admits a legal system of moves, and therefore the associated right-angled Coxeter group $G_{n+1}$ is Gromov hyperbolic and virtually fibers. Generalizing and refining the Main Theorem in \cite{osajda}, we show that $G_{n+1}$ has cohomological dimension $n+1$.
    Let $D_{n+1}$ denote the Davis complex for $G_{n+1}$.

    \item Take a quotient of $D_{n+1}$ by a torsion-free finite index subgroup of $G_{n+1}$ of sufficiently high index, and call the resulting quotient cube complex $X_{n+1}$. We verify that $X_{n+1}$ is $5$-large, has no disconnecting cubes and no isolated corners, and has cohomological dimension $n+1$.    

    \item Return to Step 1, replacing $X_n$ with $X_{n+1}$.
\end{enumerate}

\begin{rmk}
    For each $k \geq 5$, one can start with a $k$-cycle for $X_1$. Since $\alpha$-thickening only increases the number of vertices, the resulting right-angled Coxeter group with vcd $n$ has abelianization of order at least $2^k$. Letting $k$ go to infinity, one obtains examples in each vcd with abelianization of arbitrarily large order. In particular, one obtains infinitely many isomorphism classes with each fixed vcd.
\end{rmk}

\begin{rmk}
    The reader might naturally wonder whether probabilistic methods could be used to easily produce random graphs whose associated right-angled Coxeter groups satisfy the conclusions of our main theorem. The classical edge independent model $G(n,p)$ of random graphs considers graphs on $n$ vertices, where each edge has probability $p$ of being included. Davis and Kahle \cite{DavisKahle} showed that within the parameter range $n^{-1/k}\ll p \ll n^{-1/(k+1)}$ the right-angled Coxeter group associated to a randomly generated graph will almost surely have cohomological dimension $k$ as $n$ goes to infinity. However, within the parameter range $1/n \ll p \ll 1-\frac{1}{n^2}$ the random graph will almost surely have associated right-angled Coxeter group that is not Gromov hyperbolic (see \cite[Cor.\ 2.2]{CharneyFarber}. Thus within the random model $G(n,p)$, the random graphs that will almost surely produce right-angled Coxeter groups of cohomological dimension at least $2$ will almost surely correspond to right-angled Coxeter groups that are \textbf{not} Gromov hyperbolic. 
\end{rmk}

\subsection*{Acknowledgements}  This work was completed as part of the `Collaborative Projects in Geometric Topology' program organized by the Geometry and Topology group at the University of Virginia and supported by an RTG grant from the National Science Foundation (DMS-1839968). We thank the main organizers Sara Maloni and Thomas Mark for providing us the opportunity to work together on this project. We also thank Mike Davis, Jingyin Huang, and Boris Okun for some helpful comments and references. Lafont was partially supported by the NSF under grant DMS-2407438. Stover was partially supported by NSF grants DMS-2203555 and DMS-2506896, along with award SFI-MPS-TSM-00014184 from the Simons Foundation. The authors thank the referee for a very careful reading of the paper.

\section{Systems of Moves and Thickenings of Cube Complexes}

In what follows we use some of the general theory of CAT(0) cube complexes. See \cite{Davis, Sageev} for standard references.

\subsection{Jankiewicz--Norin--Wise States and Moves}\label{sec:JNW}

The proof that our right-angled Coxeter groups virtually fiber relies on the \emph{systems of moves} and \emph{legal states} defined and developed by Jankiewicz, Norin, and Wise in \cite{JNW}. Here we quickly discuss their terminology and the main result that we will use in this paper. In what follows $\Z_2$ will denote the cyclic group of order two.

Let $\G = \G(V)$ be a finite simplicial graph with vertex set $V$. A \emph{state} of $\G$ is a function ${\var \colon V \to \Z_2}$, which amounts to an assignment of 0 or 1 to every vertex in the graph. A state is said to be \emph{legal} if the subgraphs of $V$ induced by $\var^{-1}(0)$ and $\var^{-1}(1)$ are both connected and nonempty. A \emph{move} at $v \in V$ is a state $m_v$ with the following properties:
    \begin{enumerate}
        \item $m_v (v) = 1$
        \item  $m_v(u) = 0$ if $u$ and $v$ are adjacent in $\G$.
    \end{enumerate}
A \emph{system of moves} is a choice $m_v$ for each $v \in V$. A move is typically not a legal state, except in the case where $m_v^{-1}(1) = \{v\}$ and $v$ is not a cut point. Notice that it is allowable for two nonadjacent vertices to be assigned the same move, meaning that $m_v = m_w$ as functions from $V$ to $\Z_2$.

There is a natural bijection between the collection of states and the group $\Z_2^V$. This gives the collection of states the structure of an abelian group and, in particular, the subgroup $M \le \Z_2^V$ generated by a system of moves is well-defined. A system of moves is called \emph{legal} if there exists a legal state $S$ such that every state in the set
    \[ M \cdot S = \{ m + S\ :\ m \in M \} \]
is a legal state, and such an orbit $M \cdot S$ is called a \emph{legal orbit}. Using this setup, Jankiewicz, Norin, and Wise applied Bestvina--Brady Morse Theory \cite{BB} to prove the following result.

\begin{thm}[Theorem 4.3 and Corollary 4.4 in \cite{JNW}]\label{thm:JNW}
    Suppose that $\G$ is a finite graph on which there exists a legal system of moves. Then the right-angled Coxeter group $G$ with defining graph $\G$ virtually algebraically fibers.  
\end{thm}

\subsection{Simplicial Thickenings of Cube Complexes}\label{sec:alpha-thickening}

Let $X$ denote a cube complex. Osajda \cite{osajda} defines the \emph{thickening} $\Th(X)$ of $X$ to be the simplicial complex whose vertices are the same as the vertices of $X$, and vertices $v_1, \hdots, v_k$ of $\Th(X)$ span a simplex if and only if the vertices $v_1 , \hdots, v_k$ are contained in a common cube of $X$. Osajda's thickening is a special case of our more general $\alpha$-thickening discussed below. To distinguish between the ideas, we refer to Osajda's thickening using the notation $\Th_1(X)$.

We recall some terminology related to simplicial complexes. A subcomplex $\Sigma\prime \subset \Sigma$ is \emph{full} if every simplex of $\Sigma$ whose vertices lie in $\Sigma\prime$ is actually contained in $\Sigma\prime$. A simplicial complex $\Sigma$ is \emph{flag} if every $(k+1)$-tuple of pairwise adjacent vertices spans a $k$-simplex. Flag simplicial complexes are completely determined by their $1$-skeleton. For $k\geq 5$, we say a simplicial complex $\Sigma$ is \emph{$k$-large} if it is flag, and any embedded cycle of length $<k$ fails to be full. Note that, for $k=5$, this implies that for any embedded $4$-cycle $(v_1, v_2, v_3, v_4)$, one of the diagonals must be present, i.e. either $v_1$ is joined to $v_3$ by an edge, or $v_2$ is joined to $v_4$. The property of being $5$-large is sometimes also called the ``flag-no-square'' condition. 

A simplicial complex or cube complex is called \emph{locally $k$-large} if the link of every vertex is a $k$-large simplicial complex. If $X$ is a locally $5$-large cubical complex, then $\Th_1(X)$ is also a locally $k$-large simplicial complex \cite[Lem.\ 3.2]{osajda}. Since we will often have to keep track of whether or not vertices lie in a common cube, we will make use of the following:

\begin{defn}
    Two distinct vertices $v$ and $w$ in a cube complex $X$ have \emph{cubical distance} 1 if and only if they are contained in a common cube. More generally, the \emph{cubical distance} $d_{\text{cube}}(v,w)$ between distinct vertices $v$ and $w$ is the minimum number $n$ such that there is a sequence of vertices ${v=v_0,\ldots,v_n=w}$ where $d_{\text{cube}}(v_i,v_{i+1})=1$ for each $i$.
\end{defn}


Given a vertex $v$ in a cube complex $X$, the cubical $k$-neighborhood of $v$ is the subcomplex spanned by all vertices at cubical distance $\leq k$ from $v$. It will be convenient for our purposes to introduce the following terminology:

\begin{defn}
    Let $X$ be a cube complex. We say $X$ is \emph{$5$-large} if it is locally $5$-large, and has the property that the cubical $2$-neighborhood of every vertex is contractible. 
\end{defn}

Note that any locally $5$-large cube complex is locally CAT($0$) by Gromov's Lemma (see \cite[App.\ I]{Davis}). It follows that a simply connected locally $5$-large cube complex is CAT($0$). In particular, if $X$ is a $5$-large cube complex, then the universal cover $\tilde X$ is a CAT($0$) cube complex. The hypothesis on cubical $2$-neighborhoods allows us to identify such neighborhoods in $X$ with corresponding neighborhoods in the universal cover $\tilde X$.

In order to construct a legal system of moves as defined in \Cref{sec:JNW}, we will need a larger thickening than the one introduced by Osajda. The definition is as follows.  

\begin{defn}
    Let $X$ be a finite cube complex with vertex set $V(X)$. Given a finite set $Y$ equipped with a surjective map $\alpha \colon Y \to V(X)$, the \emph{$\alpha$-thickening} of $X$ is the unique simplicial complex $\Th_\alpha(X)$ with vertex set $Y$ so that $y_0, \dots, y_k \in Y$ span a $k$-simplex of $\Th_\alpha(X)$ if and only if the images ${\alpha(y_0), \dots, \alpha(y_k)}$ are contained in a common cube of $X$.
\end{defn}

\begin{figure}
\centering
\begin{tikzpicture}
  \def\RAD{1.5};
  \def\rad{0.5};
  \begin{scope}[shift={(-3.25,0)}]
    \foreach \i in {1,2,3,4,5}{
      \draw[thick,black] ({72*(\i-1)}:\RAD) -- ({72*\i)}:\RAD);
    }
    \foreach \i in {1,2,3,4,5}{
      \draw[thick,cbred, fill=cbred] ({72*(\i-1)}:\RAD) circle[radius=2pt];
    }
%
  \end{scope}
  \draw[ultra thick, dashed,-latex] (0.75,0) to node[anchor=south] {{\footnotesize $\alpha$ map}} (-0.75,0);
  \begin{scope}[shift={(3.25,0)}]
    \coordinate (a1) at ({\RAD*cos(0) + \rad*cos(0)},{\RAD*sin(0) + \rad*sin(0)});
    \coordinate (a2) at ({\RAD*cos(0) + \rad*cos(90)},{\RAD*sin(0) + \rad*sin(90)});
    \coordinate (a3) at ({\RAD*cos(0) + \rad*cos(180)},{\RAD*sin(0) + \rad*sin(180)});
    \coordinate (a4) at ({\RAD*cos(0) + \rad*cos(270)},{\RAD*sin(0) + \rad*sin(270)});
    \coordinate (b1) at ({\RAD*cos(72) + \rad*cos(0)},{\RAD*sin(72) + \rad*sin(0)});
    \coordinate (b2) at ({\RAD*cos(72) + \rad*cos(120)},{\RAD*sin(72) + \rad*sin(120)});
    \coordinate (b3) at ({\RAD*cos(72) + \rad*cos(240)},{\RAD*sin(72) + \rad*sin(240)});
    \coordinate (c1) at ({\RAD*cos(144) + \rad*cos(0)},{\RAD*sin(144) + \rad*sin(0)});
    \coordinate (c2) at ({\RAD*cos(144) + \rad*cos(60)},{\RAD*sin(144) + \rad*sin(60)});
    \coordinate (c3) at ({\RAD*cos(144) + \rad*cos(120)},{\RAD*sin(144) + \rad*sin(120)});
    \coordinate (c4) at ({\RAD*cos(144) + \rad*cos(180)},{\RAD*sin(144) + \rad*sin(180)});
    \coordinate (c5) at ({\RAD*cos(144) + \rad*cos(240)},{\RAD*sin(144) + \rad*sin(240)});
    \coordinate (c6) at ({\RAD*cos(144) + \rad*cos(300)},{\RAD*sin(144) + \rad*sin(300)});
    \coordinate (x41) at ({\RAD*cos(216) + \rad*cos(216+0)},{\RAD*sin(216) + \rad*sin(216+0)});
    \coordinate (y41) at ({\RAD*cos(216) + \rad*cos(216+90)},{\RAD*sin(216) + \rad*sin(216+90)});
    \coordinate (x42) at ({\RAD*cos(216) + \rad*cos(216+180)},{\RAD*sin(216) + \rad*sin(216+180)});
    \coordinate (y42) at ({\RAD*cos(216) + \rad*cos(216+270)},{\RAD*sin(216) + \rad*sin(216+270)});
    \coordinate (x52) at ({\RAD*cos(288) + \rad*cos(288+0)},{\RAD*sin(288) + \rad*sin(288+0)});
    \coordinate (y52) at ({\RAD*cos(288) + \rad*cos(288+90)},{\RAD*sin(288) + \rad*sin(288+90)});
    \coordinate (x53) at ({\RAD*cos(288) + \rad*cos(288+180)},{\RAD*sin(288) + \rad*sin(288+180)});
    \coordinate (y53) at ({\RAD*cos(288) + \rad*cos(288+270)},{\RAD*sin(288) + \rad*sin(288+270)});
    \foreach \A in {a1,a2,a3,a4,b1,b2,b3}{
      \foreach \B in {a1,a2,a3,a4,b1,b2,b3}{
        \draw (\A) -- (\B);
      }
    }
    \foreach \A in {b1,b2,b3,c1,c2,c3,c4,c5,c6}{
      \foreach \B in {b1,b2,b3,c1,c2,c3,c4,c5,c6}{
        \draw (\A) -- (\B);
      }
    }
    \foreach \A in {c1,c2,c3,c4,c5,c6,x41,y41,x42,y42}{
      \foreach \B in {c1,c2,c3,c4,c5,c6,x41,y41,x42,y42}{
        \draw (\A) -- (\B);
      }
    }
    \foreach \A in {x41,y41,x42,y42,x52,y52,x53,y53}{
      \foreach \B in {x41,y41,x42,y42,x52,y52,x53,y53}{
        \draw (\A) -- (\B);
      }
    }
    \foreach \A in {x52,y52,x53,y53,a1,a2,a3,a4}{
      \foreach \B in {x52,y52,x53,y53,a1,a2,a3,a4}{
        \draw (\A) -- (\B);
      }
    }
    \foreach \ang in {-144, -72, 0, 72, 144}{
      \draw[white,opacity=0.75,fill=white,fill opacity=0.75] (\ang:\RAD) circle[radius=1.5*\rad];
    }
    \foreach \A in {a1,a2,a3,a4}{
      \foreach \B in {a1,a2,a3,a4}{
        \draw[thick,cbred] (\A) -- (\B);
      }
    }
    \foreach \A in {b1,b2,b3}{
      \foreach \B in {b1,b2,b3}{
        \draw[thick,cbred] (\A) -- (\B);
      }
    }
    \foreach \A in {c1,c2,c3,c4,c5,c6}{
      \foreach \B in {c1,c2,c3,c4,c5,c6}{
        \draw[thick,cbred] (\A) -- (\B);
      }
    }
    \foreach \A in {x41,y41,x42,y42}{
      \foreach \B in {x41,y41,x42,y42}{
        \draw[thick,cbred] (\A) -- (\B);
      }
    }
    \foreach \A in {x52,y52,x53,y53}{
      \foreach \B in {x52,y52,x53,y53}{
        \draw[thick,cbred] (\A) -- (\B);
      }
    }
    \foreach \A in {a1,a2,a3,a4}{
      \draw[thick,cbred, fill=cbred] (\A) circle[radius=2pt];
    }
    \foreach \A in {b1,b2,b3}{
      \draw[thick,cbred, fill=cbred] (\A) circle[radius=2pt];
    }
    \foreach \A in {c1,c2,c3,c4,c5,c6}{
      \draw[thick,cbred, fill=cbred] (\A) circle[radius=2pt];
    }
    \foreach \A in {x41,y41,x42,y42}{
      \draw[thick,cbred, fill=cbred] (\A) circle[radius=2pt];
    }
    \foreach \A in {x52,y52,x53,y53}{
      \draw[thick,cbred, fill=cbred] (\A) circle[radius=2pt];
    }
  \end{scope}
\end{tikzpicture}
  \captionsetup{margin=0.75in}
  \captionof{figure}{A cube complex $X$ on the left and a generic $\alpha$-thickening $\Th_{\alpha}(X)$ on the right. The $\alpha$-thickenings can become extremely combinatorially complicated.}
  \label{fig:generic-thickening}
\end{figure}
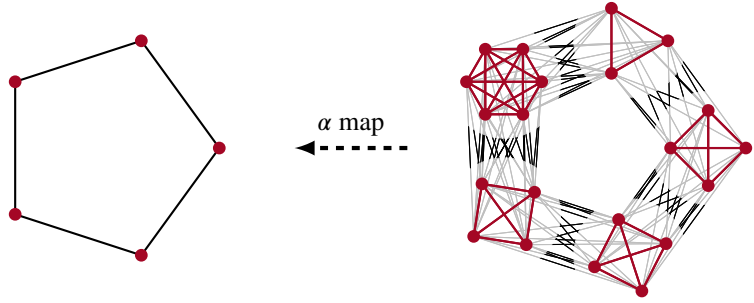

Note that the vertex set $Y$ is implicit in our notation $\Th_\alpha(X)$ for the $\alpha$-thickening. Two special cases of the thickening are noteworthy:

\begin{itemize}
    \item If $X = \{x\}$ and $Y = \{y_0, \ldots, y_n\}$, then $\alpha \colon Y \to X$ is the constant map and $\Th_{\alpha}(X)$ is a single $n$-simplex.
    \item On any cube complex $X$, we can recover the thickening $\Th_1(X)$ defined by Osajda \cite{osajda}, as it coincides with the $\alpha$-thickening associated with the identity map ${\alpha = \mathrm{Id} \colon V(X) \to V(X)}$.
\end{itemize}

We also note that there is an embedding of the Osajda thickening $\Th_1(X) \hookrightarrow \Th_\alpha(X)$ into the $\alpha$-thickening associated with any section of the map $\alpha \colon Y\rightarrow V(X)$. Moreover, $\Th_\alpha(X)$ deformation retracts to the image of $\Th_1(X)$, and hence we have a homotopy equivalence $\Th_\alpha(X)\simeq \Th_1(X)$. Combining this with the homotopy equivalence $\Th_1(X)\simeq X$ (see \cite[Lem.\ 3.5]{osajda}) we see that $\Th_\alpha(X)\simeq X$. Next we show that the thickening inherits $5$-largeness.

\begin{lem}\label{lem:flag 5-large}
    Suppose $X$ is a $5$-large cubical complex. Then for any surjective $\alpha \colon Y \to V(X)$, the simplicial complex $\Th_\alpha(X)$ is 5-large.
\end{lem}

\begin{proof}
Consider the universal cover $\tilde X$ of $X$. Since $\tilde X$ is a locally $5$-large, simply connected cube complex, we have that $\Th_1(\tilde X)$ is $5$-large; see \cite[Prop.\ 3.4]{osajda}. The covering map $\tilde X \rightarrow X$ naturally induces a covering map $\Th_1(\tilde X) \rightarrow \Th_1(X)$. Our hypothesis guarantees that cubical $2$-neighborhoods in $X$ lift to cubical $2$-neighborhoods in $\tilde X$. It follows that the combinatorial $2$-neighborhoods in $\Th_1(X)$ lift isomorphically to combinatorial $2$-neighborhoods in $\Th_1(\tilde X)$. However, the properties of being flag and $5$-large are local: they only involve cycles of length at most $4$. Any such cycle is contained in a combinatorial $2$-neighborhood, so $\Th_1(X)$ is also $5$-large. For the general case where $\Th_\alpha(X)$ is an arbitrary thickening, it is straightforward to use the embeddings $\Th_1(X) \hookrightarrow \Th_\alpha(X)$ to reduce to the $\Th_1(X)$ case. We leave the details to the reader.
\end{proof}

\section{Thickenings that produce legal systems of moves}\label{sec:thickening}

This section contains the construction used to prove the Main Theorem. Throughout this section, $X$ will denote a connected finite cube complex.

\subsection[3.1]{Thickening $\mathbb{T}$, states, and moves}\label{sec:BarryTime}

We begin by defining a particular thickening $\mathbb{T}(X)$ of a cube complex $X$ (shown in \Cref{fig:BarryX-labeling}) that, when applied to a cube complex $X$ with the appropriate properties, will admit a legal system of moves. 

\begin{defn}\label{def:state}
    Let $X$ be a cube complex for which every vertex is at cubical distance $2$ from some other vertex. The thickening $\mathbb{T}(X)$ has vertices given by the following subset of $V(X) \times V(X)$:
    \[ V(\mathbb{T}(X)) =\left\{(v,w) \; : \; d_{\text{cube}}(v,w) \geq 2 \, \text{ in } X \right\} \]
    The map $\alpha \colon V(\mathbb{T}(X)) \rightarrow V(X)$ is simply the projection onto the first factor.
\end{defn}

In other words, $V(\mathbb{T}(X))$ consists of ordered pairs of vertices $v, w$ of $X$ that are not contained in a common cube. To simplify notation, we will denote the second coordinate as a subscript, so that the $\alpha$-map just forgets the subscript, i.e., for the vertex $v_w\in V(\mathbb{T}(X))$ one has $\alpha(v_w)=v$. Also, if the vertex in the subscript is irrelevant, we will sometimes use the notation $\overline{v}$ to denote a generic element of  $\a^{-1}(v)$.

\begin{figure}[h]
\centering
\begin{tikzpicture}
    \newcommand{\cubecoords}{ 
    \coordinate (v0) at ({\scalar},0);
    \coordinate (v1) at ({\scalar*2},0);
    \coordinate (v2) at ({\scalar*4},0);
    \coordinate (v3) at ({\scalar*6},0);
    \coordinate (v4) at ({\scalar*8},0);
    \coordinate (v5) at ({\scalar*4},{\scalar*2});
    \coordinate (v6) at ({\scalar*6},{\scalar*2});
    \coordinate (v7) at ({\scalar*9},0);
    }
  \draw[ultra thick, black, dashed,-latex] (0.75,0) to node[anchor=south] {{\footnotesize $\alpha$ map}} (-0.75,0);
  \begin{scope}[shift={(-7.5,-1)}]
  \newcommand{\scalar}{0.75}
    \cubecoords
    \draw[thick, black, fill=gray, fill opacity=0.4] (v2) -- (v3) -- (v6) -- (v5) -- cycle;
    \draw[thick, black] (v1) -- (v2) -- (v3) -- (v4);
    \draw[thick, black, dotted] (v0) -- (v1);
    \draw[thick, black, dotted] (v4) -- (v7);
    \foreach \v in {v2,v3,v4,v5}{
      \draw[thick,black,fill=black] (\v) circle[radius=2pt];
    }
    \foreach \v in {v1,v6}{
      \draw[thick,cbred,fill=cbred] (\v) circle[radius=2pt];
    }
    \draw (v1) node[cbred,yshift=-1em] {{\footnotesize $v_1$}};
    \draw (v2) node[yshift=-1em] {{\footnotesize $v_2$}};
    \draw (v3) node[yshift=-1em] {{\footnotesize $v_3$}};
    \draw (v4) node[yshift=-1em] {{\footnotesize $v_4$}};
    \draw (v5) node[yshift=1em] {{\footnotesize $v_5$}};
    \draw (v6) node[cbred,yshift=1em] {{\footnotesize $v_6$}};
  \end{scope}
  \begin{scope}[shift={(0.5,-1)}]
  \newcommand{\RAD}{0.75};
  \newcommand{\rad}{0.3};
  \newcommand{\labdist}{3em};
  \newcommand{\scalar}{1}
  \cubecoords
  \begin{scope}[shift={(v1)}]
    \coordinate (v1v3) at ({360*(0/5)}:{\RAD});
    \coordinate (v1v4) at ({360*(1/5)}:{\RAD});
    \coordinate (v1v5) at ({360*(2/5)}:{\RAD});
    \coordinate (v1v6) at ({360*(3/5)}:{\RAD});
    \coordinate (v1dots) at ({360*(4/5)}:{\RAD});
  \end{scope}
  \begin{scope}[shift={(v2)}]
    \coordinate (v2v4) at ({360*(3/8)}:{\rad});
    \coordinate (v2dots) at ({360*(1/2+3/8)}:{\rad});
  \end{scope}
  \begin{scope}[shift={(v3)}]
    \coordinate (v3v1) at ({360*(1/8)}:{\rad});
    \coordinate (v3dots) at ({360*(1/2+1/8)}:{\rad});
  \end{scope}
  \begin{scope}[shift={(v4)}]
    \coordinate (v4v1) at ({360*(1/5)}:{\rad});
    \coordinate (v4v2) at ({360*(2/5)}:{\rad});
    \coordinate (v4v5) at ({360*(3/5)}:{\rad});
    \coordinate (v4v6) at ({360*(4/5)}:{\rad});
    \coordinate (v4dots) at ({360*(5/5)}:{\rad});
  \end{scope}
  \begin{scope}[shift={(v5)}]
    \coordinate (v5v1) at ({360*(-1/9+1/3)}:{\rad});
    \coordinate (v5v4) at ({360*(-1/9+2/3)}:{\rad});
    \coordinate (v5dots) at ({360*(-1/9+3/3)}:{\rad});
  \end{scope}
  \begin{scope}[shift={(v6)}]
    \coordinate (v6v1) at ({360*(1/3)}:{\RAD});
    \coordinate (v6v4) at ({360*(2/3)}:{\RAD});
    \coordinate (v6dots) at ({360*(3/3)}:{\RAD});
  \end{scope}
  \draw[gray, fill=gray, fill opacity=0.2] (v2dots) -- (v3dots) -- (v3v1) -- (v6dots) -- (v6v1) -- (v5v1) -- (v5v4) -- (v2v4) -- cycle;
  \foreach \v in {v2v4,v2dots,v3v1,v3dots,v5v1,v5v4,v5dots,v6v1,v6v4,v6dots}{
    \foreach \w in {v2v4,v2dots,v3v1,v3dots,v5v1,v5v4,v5dots,v6v1,v6v4,v6dots}{
      \draw[gray] (\v) -- (\w);
    }
  }  
  \draw[gray, fill=gray, fill opacity=0.2] (v2v4) -- (v2dots) -- (v1dots) -- (v1v6) -- (v1v5) -- (v1v4) -- cycle;
  \foreach \v in {v1v3,v1v4,v1v5,v1v6,v1dots}{
    \foreach \w in {v2v4,v2dots}{
      \draw[black] (\v) -- (\w);
    }
  }
  \foreach \v in {v2v4,v2dots}{
    \foreach \w in {v3v1,v3dots}{
      \draw[black] (\v) -- (\w);
    }
  }
  \foreach \v in {v2v4,v2dots}{
    \foreach \w in {v5v1,v5v4,v5dots}{
      \draw[black] (\v) -- (\w);
    }
  }
  \draw[gray, fill=gray, fill opacity=0.2] (v3v1) -- (v3dots) -- (v4v6) -- (v4dots) -- (v4v1) -- cycle;
  \foreach \v in {v3v1,v3dots}{
    \foreach \w in {v4v1,v4v2,v4v5,v4v6,v4dots}{
      \draw[black] (\v) -- (\w);
    }
  }
  \foreach \v in {v3v1,v3dots}{
    \foreach \w in {v6v1,v6v4,v6dots}{
      \draw[black] (\v) -- (\w);
    }
  }
  \foreach \v in {v5v1,v5v4,v5dots}{
    \foreach \w in {v6v1,v6v4,v6dots}{
      \draw[black] (\v) -- (\w);
    }
  }
  \foreach \v in {v2v4,v2dots}{
    \foreach \w in {v2v4,v2dots}{
      \draw[thick,black] (\v) -- (\w);
    }
  }
  \foreach \v in {v3v1,v3dots}{
    \foreach \w in {v3v1,v3dots}{
      \draw[thick,black] (\v) -- (\w);
    }
  }
  \foreach \v in {v4v1,v4v2,v4v5,v4v6,v4dots}{
    \foreach \w in {v4v1,v4v2,v4v5,v4v6,v4dots}{
      \draw[thick,black] (\v) -- (\w);
    }
  }
  \foreach \v in {v5v1,v5v4,v5dots}{
    \foreach \w in {v5v1,v5v4,v5dots}{
      \draw[thick,black] (\v) -- (\w);
    }
  }
  \draw[white,opacity=0.75,fill=white,fill opacity=0.75] (v1) circle[radius=1.5];
  \draw[cbred, fill=cbred, fill opacity=0.2] (v1v3) -- (v1v4) -- (v1v5) -- (v1v6) -- (v1dots) -- cycle;
  \begin{scope}[shift={(v1)}]
    \draw ({360*(0/5)}:{\RAD+\labdist}) node[rotate=90,color=cbred] {{\footnotesize $\boldsymbol{\cdots}$}};
    \draw ({360*(1/5)}:{\RAD+\labdist}) node[color=cbred] {{\footnotesize $(v_1,v_3)$}};
    \draw ({360*(2/5)}:{\RAD+\labdist}) node[color=cbred] {{\footnotesize $(v_1,v_4)$}};
    \draw ({360*(3/5)}:{\RAD+\labdist}) node[color=cbred] {{\footnotesize $(v_1,v_5)$}};
    \draw ({360*(4/5)}:{\RAD+\labdist}) node[color=cbred] {{\footnotesize $(v_1,v_6)$}};
    \foreach \v in {v1v3,v1v4,v1v5,v1v6,v1dots}{
      \foreach \w in {v1v3,v1v4,v1v5,v1v6,v1dots}{
        \draw[thick,cbred] (\v) -- (\w);
      }
    }
    \foreach \v in {v1v3,v1v4,v1v5,v1v6,v1dots}{
      \draw[thick, cbred, fill=cbred] (\v) circle[radius=2pt];
    }
  \end{scope}
  \draw[white,opacity=0.75,fill=white,fill opacity=0.75] (v6) circle[radius=1.5];
  \draw[cbred, fill=cbred, fill opacity=0.2] (v6v1) -- (v6v4) -- (v6dots) -- cycle;
  \begin{scope}[shift={(v6)}]
    \draw ({360*(1/3)}:{\RAD+\labdist}) node[color=cbred] {{\footnotesize $(v_6, v_4)$}};
    \draw ({360*(2/3)}:{\RAD+\labdist}) node[rotate=-30,color=cbred] {{\footnotesize $\boldsymbol{\cdots}$}};
    \draw ({360*(3/3)}:{\RAD+\labdist}) node[color=cbred] {{\footnotesize $(v_6,v_1)$}};
    \foreach \v in {v6v1,v6v4,v6dots}{
      \foreach \w in {v6v1,v6v4,v6dots}{
        \draw[thick,cbred] (\v) -- (\w);
        }
    }
    \foreach \v in {v6v1,v6v4,v6dots}{
      \draw[thick, cbred, fill=cbred] (\v) circle[radius=2pt];
    }
  \end{scope}
  \end{scope}
\end{tikzpicture}
    \captionsetup{margin=0.75in}
    \captionof{figure}{A cube complex $X$ on the left and the thickening $\mathbb{T}(X)$ on the right, highlighting the labeling of the preimages $\alpha^{-1}(v_1)$ and $\alpha^{-1}(v_6)$.}
    \label{fig:BarryX-labeling}
\end{figure}
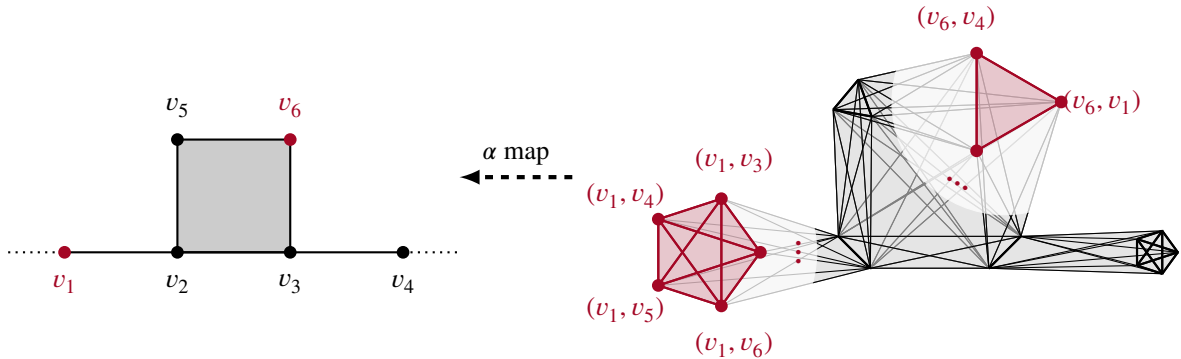

Next we define the initial state on $\mathbb{T}(X)$. We start by fixing once and for all a reference indexing $V(X) =\{ x_1, \ldots , x_k\}$ of the vertices of $X$. In terms of this reference indexing, we then assign 
\begin{equation}\label{eqn:starting state}
\mathcal{S}(v_w)=1 \text{ if and only if } v=x_i, \, w=x_j \text{ and } i<j. 
\end{equation}
 
Suppose $v, w \in V(X)$ with $d_{\text{cube}}(v,w) \geq 2$, so there are associated vertices $v_w, w_v \in \mathbb{T}(X)$. The key property of our starting state $\mathcal{S}$ is that it assigns opposite values to these two vertices. If $v=x_i$, $w = x_j$, and $ i < j$, then $\S(v_w) = 1$ and $\S(w_v) = 0$. We now define the system of moves that, when paired with the state $\S$ described in \eqref{eqn:starting state}, produce a legal system on $\mathbb{T}(X)$.
 
\begin{defn}\label{def:moves}
    Given any pair of vertices $v,w \in X$ at cubical distance $\geq 2$, we define the function $m_{vw} \colon V(\mathbb{T}(X))\rightarrow \Z_2$ by setting $m_{vw}(v_w)=m_{vw}(w_v)=1$ and $m_{vw}$ to be identically zero on all other vertices. 
\end{defn}

 
Note that the vertices $v_w$ and $w_v$ are not adjacent in $\mathbb{T}(X)$, since they are at cubical distance at least $2$ in $X$. This implies that each function $m_{vw}$ is a move at the vertices $v_w$ and $w_v$ (and in fact, $m_{vw}=m_{wv}$). Informally, the move $m_{vw}$ toggles the $\mathbb{Z}_2$-assignment of both vertices $v_w$ and $w_v$, and preserves the $\mathbb{Z}_2$-assignment of all other vertices. Let $M$ denote the subgroup generated by the system of moves.

\begin{rmk}\label{rem:opposite parity}
    Note that $m \S (v_w) \neq m \S(w_v)$ for any $m \in M$ and vertices $v, w \in V(X)$ with $d_{\text{cube}}(v,w) \geq 2$. Indeed, the initial state satisfies $\S(v_w) \neq \S(w_v)$, and every generator of $M$ assigns the same value to $v_w$ and $w_v$, and hence keeps the $\mathbb{Z}_2$-assignment of $v_w$ and $w_v$ different.
\end{rmk}

\subsection{Verifying the system is legal}

The thickening $\mathbb{T}(X)$ can be applied to any finite cube complex $X$ for which every vertex is at cubical distance $2$ from some other vertex. We will now find suitable conditions on $X$ to guarantee the initial state and moves on $\mathbb{T}(X)$ provide a legal system. The required assumptions on $X$ are contained in the next two definitions.

\begin{defn}
    A cube complex $X$ has \emph{no isolated corners} if, given any cube $\square \subset X$ and vertex $v\in \square$, there exists a vertex $w$ adjacent to $v$ and satisfying $w\notin \square$.
    $X$ has $\emph{no disconnecting cubes}$ if, for every $\square \subset X$, the full subcomplex with vertex set $V(X) \setminus V(\square)$ is connected.
\end{defn}

Our goal in this section will be to establish the following.

\begin{prop}\label{prop:BarryX-has-legal-system}
    Let $X$ be a $5$-large cube complex. If $X$ has no disconnecting cubes and no isolated corners, then the initial state $\mathcal{S}$ and system of moves $\{ m_{vw}\}$ defined in this section form a legal system of moves on the $1$-skeleton of $\mathbb{T}(X)$.
\end{prop}

The following definition and lemma will be used in the proof of this result.  

\begin{defn}
    Fix a cube complex $X$, a thickening Th$_\a(X)$, a state $S$ on $V(\text{Th}_\a(X))$, and $\epsilon \in \Z_2$. We define the \emph{detection function} corresponding to $S$ by
    \[ \delta_{\epsilon} \colon V(X) \longrightarrow \{ \text{Y}, \text{N} \} \]
    as follows. If there exists $\overline{v} \in \a^{-1}(v)$ with $S \left( \overline{v} \right) = \epsilon$, then $\delta_{\epsilon}(v) = \text{Y}$. Otherwise $\delta_{\epsilon}(v) = \text{N}$.
\end{defn}

The function $\delta_{\epsilon}$ detects whether or not there is a vertex in the $\a$-preimage of $v$ that is assigned a value of $\epsilon$ by $S$. If such a vertex exists, then $\delta_\epsilon$ assigns a value of Y for ``yes''. Otherwise, it assigns N for ``no''. Also note that, while we suppress this in the notation, the detection function $\delta_\epsilon$ depends on the state $S$. In what follows the given state will always be clear from context.

\begin{lem}\label{lem:maximal cube}
    Let $X$ be a $5$-large cube complex with no isolated corners and no disconnecting cubes, and let $\mathbb{T}(X)$ be the associated thickening.
    Fix $m \in M$ and $\epsilon \in \Z_2$. Let $\delta_\epsilon$ denote the detection function corresponding to the state $m \mathcal{S}$. Then there exists a cube $\square_m$ in $X$ such that $v \notin \square_m$ implies $\delta_{\epsilon}(v) = Y$.
\end{lem}

Note that the cube $\square_m$ may not be unique. In fact, $\delta_{\epsilon}(v)$ could be $\text{Y}$ for all $v \in V(X)$, in which case any cube in $X$ vacuously satisfies the conclusion of the lemma. Further, it is possible that $\delta_{\epsilon}(v) = \text{Y}$ for some (or even all) of the vertices $v \in \square_m$.

\begin{proof}[Proof of \Cref{lem:maximal cube}]
Suppose $v \in V(X)$ satisfies $\delta_\epsilon(v) = \text{N}$, and let $w \in V(X)$ be any vertex with $d_{\text{cube}}(v,w) \geq 2$. Then there are associated vertices $v_w$ and $w_v$ in $V(\mathbb{T}(X))$, and \Cref{rem:opposite parity} implies that $m \mathcal{S}(v_w) \neq m \mathcal{S}(w_v)$. Since $\delta_\epsilon(v) = \text{N}$, this means by definition that $m \mathcal{S} (v_w) \neq \epsilon$, and therefore $m \mathcal{S} (w_v) = \epsilon$ and so $\delta_{\epsilon}(w) = \text{Y}$. This shows that all vertices $x$ with $\delta_\epsilon(x) = \text{N}$ must be at cubical distance $1$ from each other.

It now remains to argue that all such vertices $\{v_1, \ldots , v_k\}$ lie in a common cube in $X$. We know from \Cref{lem:flag 5-large} that $\mathbb{T}(X)$ is $5$-large, and hence flag. By the above reasoning, the preimages of all the $v_i$ form a collection of pairwise adjacent vertices in $\mathbb{T}(X)$. Since $\mathbb{T}(X)$ is flag, this collection of vertices span a simplex in $\mathbb{T}(X)$. By the definition of a thickening, this implies the vertices $\{v_1, \ldots , v_k\}$ are all contained in a common cube of $X$.
\end{proof}

With the Lemma in hand, we can now establish \Cref{prop:BarryX-has-legal-system}.

\begin{proof}[Proof of \Cref{prop:BarryX-has-legal-system}]
Choose $m \in M$ and $\epsilon \in \Z_2$, and let $\square_m$ be the cube provided by \Cref{lem:maximal cube}. Then fix $\overline{x}, \overline{y} \in V(\mathbb{T}(X))$ such that $m \mathcal{S} \left( \overline{x} \right) = m \mathcal{S} \left( \overline{y} \right) = \epsilon$. We must show that $\overline{x}$ and $\overline{y}$ are connected by a path within the subgraph induced by the set of vertices $\left( m \mathcal{S} \right)^{-1} (\epsilon)$.  

Set $x=\alpha\left(\overline{x}\right)$ and $y=\alpha\left(\overline{y}\right)$. We first assume that neither $x$ nor $y$ are contained in $\square_m$. Since $X$ has no disconnecting cubes, there is a (cubical) path $x * v_1 * v_2 * \hdots * v_k * y$ that does not meet $\square_m$. Then $\delta_\epsilon (v_i) = \text{Y}$ for each $i$, since all vertices with an output of N are contained in $\square_m$. Hence, for each $i$ there exists $\overline{v}_i \in \alpha^{-1}(v_i)$ such that $m \mathcal{S} \left( \overline{v}_i \right) = \epsilon$. The path $\overline{x} * \overline{v}_1 * \hdots * \overline{v}_k * \overline{y}$ thus connects $\overline{x}$ to $\overline{y}$ via vertices in $\left( m \mathcal{S} \right)^{-1} (\epsilon)$.

Now, suppose that either $x$ or $y$ is in $\square_m$. If both $x$ and $y$ are in $\square_m$ then the desired conclusion is clear, since all vertices in $\alpha^{-1}(x)$ and $\alpha^{-1}(y)$ are connected by an edge in $\mathbb{T}(X)$. Without loss of generality assume $y$ is in $\square_m$ and $x$ is not. Since $X$ has no isolated corners, there exists $z \in V(X)$ connected to $y$ by an edge, with $z \notin \square_m$. The reasoning in the previous paragraph connects $\overline{x}$ to some vertex $\overline{z} \in \alpha^{-1}(z)$ with a path contained in $\left( m \mathcal{S} \right)^{-1} (\epsilon)$. We then append $\overline{y}$ to the end of the path since it is connected by an edge to $\overline{z}$ by assumption.
\end{proof}

\begin{rmk}
    As shown above, the assumptions that there are no disconnecting cubes and no isolated corners on $X$ are sufficient for $X$ to admit an $\alpha$-thickening with a legal system of moves.
    For large cube complexes, one would expect that these conditions are also necessary by the following reasoning. Suppose $X$ does not satisfy the no disconnecting cube condition, and let $\square$ denote a maximal cube whose removal disconnects $X$. Then one can apply a sequence of moves making every vertex $v \in \square$ satisfy $\delta_{\epsilon}(v) = \text{N}$. In this way, $\square$ would serve as an ``$\epsilon$-blockade'' between the two components of $X \setminus \square$, and thus the corresponding system of moves would not be legal. If $X$ does not satisfy the no isolated corner condition, then any such corner vertex $v \in \square$ could be similarly blocked from the rest of the cube complex, by applying a sequence of moves that change the vertices above all the vertices in $\square \setminus v$.
    For small cube complexes (say, for example, three vertices connected as a line segment), these assumptions are not quite necessary, but one would expect these conditions to be necessary for larger complexes.
\end{rmk}

\section{Proof of the Main Theorem}\label{sec:MainPf}

We now prove our main theorem following the inductive scheme outlined in the introduction. We start with a finite cube complex $X_n$ having the following four properties:
\begin{itemize}
    \item $X_n$ is $5$-large;
    \item $X_n$ has no disconnecting cubes;
    \item $X_n$ has no isolated corners;
    \item $X_n$ has cohomological dimension $n$.
\end{itemize}

\bigskip

We start our process at the initial stage $n=1$, by taking $X_1$ to be a cycle of length at least five. Note that all four of the above properties are then immediate.

Now consider the thickening $T_n\coloneqq \mathbb{T}(X_n)$ of $X_n$ defined in the preceding section. The $1$-skeleton of the simplicial complex $T_n$ defines a right-angled Coxeter group $G_{n+1}$. The proof now has the following three steps.

\bigskip

\noindent \textbf{Step 1:} Verify that the $1$-skeleton of $T_n$ has the sufficient graph theoretic properties for $G_{n+1}$ to be Gromov hyperbolic and virtually fiber.

\bigskip

\noindent \textbf{Step 2:} Check that $G_{n+1}$ has virtual cohomological dimension $n+1$.

\bigskip

\noindent \textbf{Step 3:} If
$D_{n+1}$ is the Davis complex associated with the Coxeter group $G_{n+1}$, then verify that there is a suitable torsion-free finite index subgroup $\Lambda \leq G_{n+1}$ so that $X_{n+1}=D_{n+1}/\Lambda$ satisfies the four properties in our inductive scheme.

\bigskip

The first step follows immediately from the arguments in \Cref{sec:thickening}. Indeed, \Cref{lem:flag 5-large} shows that the simplicial complex $T_n$ is $5$-large, so the Coxeter group $G_{n+1}$ will be Gromov hyperbolic. Moreover, \Cref{prop:BarryX-has-legal-system} shows that the $1$-skeleton of $T_n$ has a legal system, so the Coxeter group $G_{n+1}$ will virtually fiber by \cite[Thm.\ 6.14]{JNW}. This completes verification of the first step.

For Step 2, we want to compute the virtual cohomological dimension of the right-angled Coxeter group $G_{n+1}$ associated with the complex $T_n$. A result of Davis (see \cite[Cor.\ 8.5.5]{Davis}) allows us to compute the virtual cohomological dimension from the topology of the simplicial complex $T_n$:
\begin{equation}\label{eqn:vcd def}
    \vcd(G_{n+1})= \max\!\left\{k\ :\ \overline{H}{}^{k-1}(T_n\setminus \sigma)\neq 0 \textrm{ for some } \sigma \subset T_n\right\}\!.
\end{equation}
To calculate the cohomological dimension of spaces of the form $T_n\setminus \sigma$ where $\sigma$ is a simplex of $T_n$, we will use a Mayer--Vietoris sequence. This will require certain auxiliary subsets of $T_n$ that we now introduce.

Given a simplicial complex $\Sigma$, we can metrize $\Sigma$ by making each simplex isometric to a spherical all right simplex of the appropriate dimension. Therefore each edge has length $\pi/2$, each triangle is isometric to the portion of the $2$-sphere in the positive octant, and so forth. Given a simplex $\sigma\subset \Sigma$, consider the set $S(\sigma, \Sigma)$ of points at distance $\epsilon>0$ from $\sigma$, thought of as a sphere centered on $\sigma$. Note that the topology of this space is independent of the choice of small positive $\epsilon>0$. For example, if $v\in \Sigma$ is any vertex, then $S(v, \Sigma)$ is homeomorphic to the link $\lk(v,\Sigma)$. However, for higher dimensional simplices, $S(\sigma, \Sigma)$ is a more complicated space built up from the links of $\sigma$ as well as the links of the faces of $\sigma$. In the next section, we will establish the following:

\begin{prop}\label{prop:cd-spherical-nbhds}
    For every nonempty simplex $\sigma \subset T_n$, we have that $\cdim \bigl(S(\sigma, T_n)\bigr)\leq n-1$.
\end{prop} 

Here the cohomological dimension $\cdim(Z)$ of a space $Z$ is the largest integer $k$ such that $\overline{H}{}^{k}(Z)$ is nontrivial, which is a slight abuse of terminology\footnote{This is not the standard definition, as cohomological dimension is usually defined as the largest $k$ so that $\overline{H}{}^{k}(Z, C) \neq 0$ for some closed subset $C\subset Z$. Throughout our paper we are working with the special case where $C=\emptyset$. Nevertheless, it will be convenient to use the notation $\cdim(Z)$, which could be smaller than the classically defined cohomological dimension of $Z$. We hope this will not cause confusion for the reader.}. The proof of \Cref{prop:cd-spherical-nbhds} will be postponed until the next section, as it is somewhat involved. Assuming this result, we can now continue our verification of the second step.

\begin{prop}\label{prop:vcd-exact}
    The group $G_{n+1}$ has virtual cohomological dimension $= n+1$.
\end{prop}

\begin{proof}
Equation \eqref{eqn:vcd def} implies the virtual cohomological dimension of $G_{n+1}$ can be computed from the complex $T_n$. In particular, its vcd is one more than the maximal cohomological dimension $\cdim(T_n\setminus \sigma)$, where $\sigma$ ranges over all simplices $\sigma \subset T_n$. In the special case where $\sigma = \emptyset$, we see that $\cdim(T_n)=\cdim(X_n) = n$, establishing the lower bound $\vcd(G_{n+1}) \geq n+1$.

To show that $\vcd(G_{n+1})= n+1$ we must verify that $\cdim(T_n \setminus \sigma)\leq n$ for each nonempty simplex $\sigma \subset T_n$. Given such a simplex $\sigma$, cover $T_n$ by two open sets $U,V$, by setting $U$ to be an $\epsilon$-neighborhood of $\sigma$, and $V$ to be the complement $T_n \setminus \sigma$. The intersection $U\cap V$ is then homotopy equivalent to $S(\sigma, T_n)$. Since $U$ is contractible, the Mayer--Vietoris sequence in reduced cohomology reduces to the long exact sequence:
    \[ \cdots \longrightarrow \overline{H}{}^i(T_n) \longrightarrow \overline{H}{}^i(T_n\setminus \sigma) \longrightarrow \overline{H}{}^i(S(\sigma, T_n)) \longrightarrow \overline{H}{}^{i+1}(T_n) \longrightarrow \cdots \]
Since $\cdim(T_n)=n$, the sequence gives us isomorphisms $\overline{H}{}^i(T_n\setminus \sigma)
\cong \overline{H}{}^i(S(\sigma, T_n))$ in all degrees $i\geq n+1$. However we know from \Cref{prop:cd-spherical-nbhds} that $\cdim(S(\sigma, T_n))\leq n-1$, so the groups $\overline{H}{}^i(T_n\setminus \sigma)$ then vanish when $i\geq n+1$. We conclude that $\cdim(T_n\setminus \sigma)\leq n$, as desired. This completes the proof of the proposition.
\end{proof}

We now proceed to the third step, namely checking that the inductive hypotheses are satisfied for a suitable quotient $X_{n+1}$ of the Davis complex $D_{n+1}$. We first prove that the analogous properties for the associated Davis complex are satisfied, then verify that they descend to a suitably chosen quotient of the Davis complex.

\begin{lem}\label{lem:davis-clx-inductive-step}
    The Davis complex $D_{n+1}$ is $5$-large, has no isolated corners, has no disconnecting cubes, and the group $G_{n+1}$ is one-ended.
\end{lem}

\begin{proof}
Links of the vertices of $D_{n+1}$ are copies of the $5$-large simplicial complex $T_n$, so local $5$-largeness is immediate. Moreover, cubical neighborhoods of vertices in CAT($0$) cube complexes (such as $D_{n+1}$) are known to be convex (see \cite[Lem.\ 13.15]{HaglundWise}) hence contractible. This verifies that $D_{n+1}$ is $5$-large. Recall that the $1$-skeleton of the Davis complex is canonically identified with the Cayley graph of the Coxeter group with respect to its standard generators. To see that $D_{n+1}$ has no isolated corners, just note that at an isolated corner, the link would have to be a single simplex. This is impossible, since the link is $T_n$, which is not contractible since $\cdim(T_n)\neq 0$. It remains to check that $D_{n+1}$ has no disconnecting cubes, and it is sufficient to argue that $G_{n+1}$ is one-ended. This can be determined from $T_n$ by verifying that the complement $T_n \setminus \sigma$ is connected for every simplex $\sigma \subset T_n$; see \cite[Thm.\ 8.7.2]{Davis}. However, this is clear since $T_n$ is a thickening of $X_n$ and, by induction, $X_n$ has no disconnecting cubes and no isolated corners. Noting that the simplex $\sigma$ lies above a cube, the argument in the proof of \Cref{prop:BarryX-has-legal-system} shows the complement of $\sigma$ is connected.
\end{proof}

\begin{cor}
    For a suitable finite index torsion-free subgroup $\Lambda \leq G_{n+1}$, the finite cube complex $X_{n+1}\coloneqq D_{n+1}/\Lambda$ is $5$-large, has no disconnecting cubes, has no isolated corners, and has cohomological dimension $=n+1$.
\end{cor}

\begin{proof}
Concerning the cohomological dimension, we have the series of equalities
    \[ \cd(X_{n+1})= \cd(\Lambda)=\vcd(G_{n+1})= n+1. \]
The first equality follows from the fact that the Davis complex $D_{n+1}$ is CAT($0$) and thus is contractible, hence $X_{n+1}$ is an Eilenberg--MacLane space $K(\Lambda,1)$. Since $\Lambda$ is a finite index torsion-free subgroup of $G_{n+1}$, the second equality is a well-known result of Serre \cite[Thm.\ F.3.4]{Davis}. The third equality follows directly from \Cref{prop:vcd-exact}. This establishes the statement on the cohomological dimension.
    
Next we observe that the properties of being locally $5$-large and having no isolated corners are \emph{purely local}, in that they can be determined by looking at the cubical $2$-neighborhood of vertices. From \Cref{lem:davis-clx-inductive-step} we see that the local requirements are verified for cubical $2$-neighborhoods of vertices in the Davis complex $D_{n+1}$. Therefore, if the $2$-cubical neighborhoods in $X_{n+1}$ lift to cubical $2$-neighborhoods in the universal cover, we would obtain the desired properties for $X_{n+1}$. It is now a very standard consequence of residual finiteness of the Coxeter group $G_{n+1}$ that there is a finite index subgroup $\Lambda$ of $G_{n+1}$ with this lifting property. Lastly, we need to verify that $X_{n+1}$ has no disconnecting cubes. This is again standard, following from \Cref{lem:davis-clx-inductive-step} since $D_{n+1}$ has no disconnecting cubes and $G_{n+1}$ is one-ended.
\end{proof}

\begin{rmk}
    Note that the cube complex $X_{n+1}$ has the property that there exists another vertex at cubical distance $\geq 2$ from $v$ for every vertex $v$. This follows from the fact that $X_{n+1}$ is a quotient of the unbounded cube complex $D_{n+1}$ and that cubical $2$-neighborhoods in $X_{n+1}$ lift to cubical $2$-neighborhoods in $D_{n+1}$. This confirms that we can apply our thickening process to $X_{n+1}$ (see \Cref{def:state}).
\end{rmk}

\section{Calculating the virtual cohomological dimension}

\subsection{Topology of links, spheres, and balls}\label{subsec:generalities}

In this subsection we establish some useful results on links, spheres, and balls in simplicial complexes. The reader who is primarily interested in the proof of \Cref{prop:cd-spherical-nbhds} can skip over this subsection, referring back to results as needed. We merely observe that the simplicial complexes $T_n$ are $5$-large (see \Cref{lem:flag 5-large}), as are all the links in $T_n$. Therefore all results in this section apply to them.

\medskip

Recall that $\lk(\sigma, \Sigma)$ is an abstract simplicial complex with vertices corresponding to simplices containing $\sigma$ as a codimension one face, and where a collection of vertices spans a simplex in $\lk(\sigma, \Sigma)$ if and only if the corresponding simplices in $\Sigma$ are contained in a common larger simplex in $\Sigma$. This abstract definition holds for general simplicial complexes. In the special case where $\Sigma$ is a flag simplicial complex, the link of a simplex $\sigma \subset \Sigma$ can be identified with the full subcomplex $\lk(\sigma, \Sigma)$ of $\Sigma$ spanned by all vertices at combinatorial distance one from all vertices of $\sigma$. Note that the property of being $5$-large and flag is inherited by full subcomplexes, and thus it is inherited by links.

Similarly, if $\square$ is a cube in the CAT(0) cube complex $X$, we can define the abstract \emph{combinatorial link} of the cube. Vertices are the cubes containing $\square$ as a codimension one face, and a collection of vertices spans a simplex if and only if the corresponding cubes lie in a common larger cube. There is also an alternative geometric interpretation of the link. We can define the \emph{geometric link} of $\square$ by choosing a point $p$ in the interior of $\square$, and looking at the space of unit tangent vectors at $p$ that are orthogonal to the cube $\square$. Thus each cube containing $\square$ contributes a single simplex to the geometric link. The two notions of link are closely related, as they give homeomorphic spaces. The combinatorial link is purely combinatorial, whereas the geometric link can also be endowed with a metric structure under which each simplex is isometric to an all right spherical simplex. We will shift between the geometric link and the combinatorial link as convenient, and denote this common topological space by ${\lk}(\square, X)$.

\begin{lem}\label{lem:links-cube-to-vertex}
    Let $X$ be an arbitrary cube complex, $\square \subset X$ be a cube, and $v\in \square$ be a vertex of the cube. The cube $\square$ defines a simplex $\tau \subset \lk(v,X)$, and there is a combinatorial isomorphism $\lk(\square, X) \cong \lk(\tau, \lk(v,X))$.
\end{lem}

\begin{figure}[h]
\centering
\begin{tikzpicture}
\newcommand{\scalar}{0.75}
\newcommand{\lrx}{(\scalar)*3} 
\newcommand{\lry}{(\scalar)*0} 
\newcommand{\fbx}{(\scalar)*1.5} 
\newcommand{\fby}{(\scalar)*1.5} 
\newcommand{\udx}{(\scalar)*0}
\newcommand{\udy}{(\scalar)*2.5}
\newcommand{\cubecoords}{ 
    \coordinate (cube1) at (0,0);
    \coordinate (cube2) at ($(cube1)+({\lrx},{\lry})$);
    \coordinate (cube3) at ($(cube2)+({\fbx},{\fby})$);
    \coordinate (cube4) at ($(cube3)-({\lrx},{\lry})$);
    \coordinate (cube5) at ($(cube1)+({\udx},{\udy})$);
    \coordinate (cube6) at ($(cube2)+({\udx},{\udy})$);
    \coordinate (cube7) at ($(cube3)+({\udx},{\udy})$);
    \coordinate (cube8) at ($(cube4)+({\udx},{\udy})$);
    \coordinate (cube9) at ($(cube2)+({\lrx},{\lry})$);
    \coordinate (cube10) at ($(cube3)+({\lrx},{\lry})$);
    \coordinate (cube11) at ($(cube7)+({\lrx},{\lry})$);
    \coordinate (cube12) at ($(cube6)+({\lrx},{\lry})$);
}
\begin{scope}[shift={({-3.5},0)}]
\cubecoords
\coordinate (6mid7) at ($0.6*(cube6)+0.4*(cube7)$);
\coordinate (link1) at ($(6mid7) + ({0.4*\lrx},{0.4*\lry})$);
\coordinate (link2) at ($(6mid7) + ({-0.4*\udx},{-0.4*\udy})$);
\coordinate (link3) at ($(6mid7) + ({-0.4*\lrx},{-0.4*\lry})$);
\draw (cube4) -- (cube3) -- (cube10);
\draw (cube4) -- (cube8);
\draw (cube3) -- (cube7);
\draw[ultra thick, cbred] (link1) -- (link2) -- (link3); 
\draw[line width=4pt,white] (cube5) -- (cube6) -- (cube12);
\draw[line width=4pt, white] (cube2) -- (cube6);
\draw[line width=4pt, white] (cube9) -- (cube12);
\draw (cube1) -- (cube5);
\draw (cube2) -- (cube6);
\draw (cube9) -- (cube12);
\draw (cube10) -- (cube11);
\draw (cube1) -- (cube2) -- (cube9);
\draw (cube5) -- (cube6) -- (cube12);
\draw (cube8) -- (cube7) -- (cube11);
\draw (cube1) -- (cube4);
\draw (cube2) -- (cube3);
\draw (cube5) -- (cube8);
\draw (cube6) -- (cube7);
\draw (cube9) -- (cube10);
\draw (cube12) -- (cube11);
\foreach \pt in {link1,link2,link3}{
    \draw[thick,cbred,fill=cbred] (\pt) circle[radius=1.75pt];
}
\draw (link2) node[rectangle, color=white, fill, fill opacity=0.6, inner sep=0.1em, outer sep=0.16em, text=cbred, text opacity=1, anchor=north west] {{\footnotesize $\lk(\square,X)$}};
\draw[ultra thick, cbblue] (cube6) -- (cube7);
\draw[thick,cbblue,fill=cbblue] (cube6) circle[radius=2pt];
\draw[thick,cbblue,fill=cbblue] (cube7) circle[radius=2pt];
\draw (cube7) node[text=cbblue,anchor=south] {$\square$};
\end{scope}
\begin{scope}[shift={(3.5,0)}]
\cubecoords
\coordinate (6mid2) at ($0.6*(cube6)+0.4*(cube2)$);
\coordinate (6mid5) at ($0.6*(cube6)+0.4*(cube5)$);
\coordinate (6mid7) at ($0.6*(cube6)+0.4*(cube7)$);
\coordinate (6mid12) at ($0.6*(cube6)+0.4*(cube12)$);
\draw (cube4) -- (cube3) -- (cube10);
\draw (cube4) -- (cube8);
\draw (cube3) -- (cube7);
\draw[cbblue] (6mid2) -- (6mid7); 
\draw[line width=4pt,white] (cube5) -- (cube6) -- (cube12);
\draw[line width=4pt, white] (cube2) -- (cube6);
\draw[line width=4pt, white] (cube9) -- (cube12);
\draw (cube1) -- (cube5);
\draw (cube2) -- (cube6);
\draw (cube9) -- (cube12);
\draw (cube10) -- (cube11);
\draw (cube1) -- (cube2) -- (cube9);
\draw (cube5) -- (cube6) -- (cube12);
\draw (cube8) -- (cube7) -- (cube11);
\draw (cube1) -- (cube4);
\draw (cube2) -- (cube3);
\draw (cube5) -- (cube8);
\draw (cube6) -- (cube7);
\draw (cube9) -- (cube10);
\draw (cube12) -- (cube11);
\draw[ultra thick, black] (cube6) -- (cube7);
\draw[thick,black,fill=black] (cube7) circle[radius=2pt];
\draw[thick,cbred,fill=cbred] (6mid7) circle[radius=2pt] node[text=cbred, anchor=south] {{\footnotesize $\tau$}};
\draw[cbblue, fill=cbblue, fill opacity=0.25] (6mid2) -- (6mid5) -- (6mid7) -- (6mid12) -- cycle;
\foreach \pt in {6mid2,6mid5,6mid12}{
         \draw[thick,cbred,fill=cbred] (\pt) circle[radius=1.75pt];
}
\draw[ultra thick, cbred] (6mid5) -- (6mid2) -- (6mid12);
\draw[ultra thick,cbblue,fill=cbblue] (cube6) circle[radius=2pt];
\draw (cube6) node[circle, outer sep=0.65em, text=cbblue, anchor=south east] {{\footnotesize $v$}};
\draw (6mid2) node[rectangle, color=white, fill, fill opacity=0.6, inner sep=0.1em, outer sep=0.16em, text=cbred, text opacity=1, anchor=north west] {{\footnotesize $\lk(\tau,\lk(v,X))$}};
\end{scope}
\end{tikzpicture}
\caption{The combinatorial isomorphism in \Cref{lem:links-cube-to-vertex}.}
\end{figure}
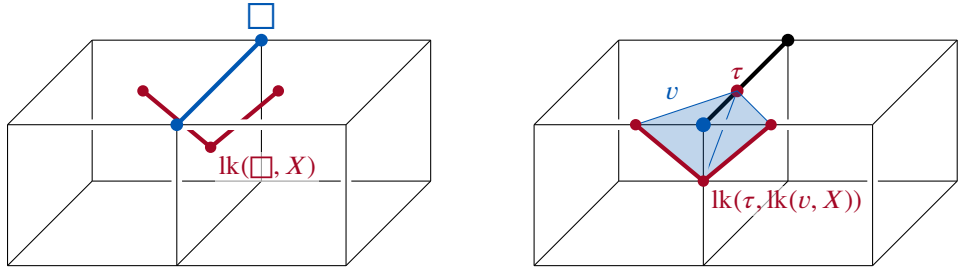

\begin{proof}
Recall that vertices in $\lk(\square, X)$ are given by cubes $\square^\prime$ satisfying $\square \subset \square^\prime$ for which $\dim(\square^\prime)=\dim(\square)+1.$ Considering the simplex $\tau \in \lk(v, X)$ corresponding to $\square$, we see that each such cube $\square^\prime$ gives rise to a simplex $\tau^\prime \subset \lk(v,X)$ containing the simplex $\tau$ as a codimension one face. In other words, this defines a vertex in $\lk(\tau, \lk(v, X))$. Conversely, a simplex $\tau^\prime$ in $\lk(v, X)$ containing $\tau$ as a codimension one face arises from a cube $\square^\prime$ containing $\square$ as a codimension one face. This establishes a bijection between the vertices of the two simplicial complexes $\lk(\square, X)$ and $\lk(\tau, \lk(v,X))$.
    
Finally, we check that the actual simplicial complexes coincide. For the simplicial complex $\lk(\square, X)$, the set of vertices $\{\square^\prime_0, \ldots , \square^\prime_k \}$ span a $k$-simplex if and only if the cubes $\{\square^\prime_0, \ldots ,\square^\prime_k\}$ are all contained in common larger cube $\square^{\prime\prime}$. Considering $\lk(v, X)$, this collection of cubes correspond to a collection of simplices $\{\tau^\prime_0, \ldots , \tau^\prime_k \}$ in $\lk(v,X)$ all containing $\tau$ as a codimension one face. The larger cube $\square^{\prime\prime}$ leads to a larger simplex $\tau^{\prime\prime} \subset \lk(v, X)$ containing all of the $\tau^\prime_i$ as faces. Viewing the $\tau^\prime_i$ as a collection of vertices in $\lk(\tau, \lk(v,X))$, it follows that they span a simplex in $\lk(\tau, \lk(v,X))$. This shows that simplices in $\lk(\square, X)$ give rise to simplices in $\lk(\tau, \lk(v,X))$. One can reverse this argument to show the converse. Thus the bijection on vertex sets described above also bijectively identifies simplices of the two complexes, establishing that the two simplicial complexes are isomorphic.
\end{proof}

An identical argument can be used to show the following result.

\begin{lem}\label{lem:links-simpclx-to-vertex}
    Let $Y$ be a simplicial complex, $\sigma \subset Y$ be a simplex, and $\tau \subset \sigma$ be a face of the simplex. Then $\sigma$ defines a simplex $\hat \sigma \subset \lk(\tau, Y)$, and there is a combinatorial isomorphism between $\lk(\sigma, Y)$ and $\lk(\hat \sigma, \lk(\tau, Y))$. In particular, $\lk(\sigma, Y) \cong \lk(\hat \sigma, \lk(v, Y))$ for any vertex $v$ of $\sigma$.
\end{lem}

Next we turn our attention to the topology of the spherical sets $S(\sigma, \Sigma)$. Recall that this set is a small $\epsilon$-sphere centered around the simplex $\sigma$, and it inherits a natural polyhedral structure from the ambient simplicial complex $\Sigma$. Indeed, each point $p\in S(\sigma, \Sigma)$ is contained in the interior of a unique simplex in $\Sigma$ that splits as a join $\sigma_p * \tau_p$ where $\sigma_p\subset \sigma$ is a face of $\sigma$ and $\tau_p$ is a simplex disjoint from $\sigma$. The intersection of $S(\sigma, \Sigma)$ with the simplex $\sigma_p * \tau_p$ is a copy of the polyhedron $\sigma_p \times \tau_p$, which we can think of as the ``$\epsilon$-slice'' $\sigma_p \times \tau_p \times \{\epsilon\}$
in the join:
    \[ \sigma_p * \tau_p = (\sigma_p \times \tau_p \times [0,1])/\sim \]
As we range over all simplices incident to $\sigma$, this yields the polyhedral decomposition of $S(\sigma, \Sigma)$. 

It will be more convenient for our purposes to replace the small $\epsilon$-sphere centered at $\sigma$ by the combinatorial $1$-sphere centered at $\sigma$. This is the full subcomplex $\hat S(\sigma, \Sigma)$ spanned by all vertices of $\Sigma$ that are at combinatorial distance one from $\sigma$, i.e., that are not in $\sigma$ but are connected by an edge to a vertex in $\sigma$.

\begin{lem}\label{lem:describe-combinatorial-sphere}
    Assume $\Sigma$ is $5$-large. Then in terms of the description of $S(\sigma, \Sigma)$ in terms of joins, $\hat S(\sigma, \Sigma)$ is the union of all the simplices $\tau_p$. 
\end{lem}

\begin{proof}
Clearly the union of the simplices $\tau_p$ is contained in $\hat S(\sigma, \Sigma)$. To show the converse, we need to argue that for any simplex $\tau \subset \hat S(\sigma, \Sigma)$, we can find a vertex $w \in \sigma$ with the property that $\tau*w$ is a simplex of $\Sigma$. This is done via induction on $\dim(\tau)$, with the base case $\dim(\tau)=0$ being immediate.

Now given $\tau$, express it as a join $v * \tau^\prime$ where $v\in \hat S(\sigma, \Sigma)$ is a vertex and $\tau^\prime \subset \hat S(\sigma, \Sigma)$ is a codimension one face. By the inductive hypothesis, there exists a vertex $w^\prime \subset \sigma$ with the property that $w^\prime * \tau^\prime$ is a simplex in $\Sigma$. If $v$ is connected to $w^\prime$, then $\tau*w^\prime$ is a simplex since $\Sigma$ is flag, and we are done. If $v$ is not connected to $w^\prime$, then $v$ is connected to some vertex $v^\prime\in \sigma$, with $v^\prime\neq w^\prime$. For each vertex $w\in \tau^\prime$, consider the ordered $4$-tuple $\{v, v^\prime, w^\prime, w\}$. This defines a square in $\Sigma$, and we know by construction that $v$ is not connected to $w^\prime$. Then $5$-largeness forces $w$ to be connected to $v^\prime$. Since this holds for arbitrary $w \in \tau^\prime$, we conclude that all the vertices of $\tau^\prime$ are connected to $v^\prime$. Again, the fact that $\Sigma$ is flag implies that $\tau * v^\prime$ is a simplex, thus completing the proof.
\end{proof}

We also introduce the notation $\hat B(\sigma, \Sigma)$, to denote the combinatorial $1$-ball centered at $\sigma$. This is the union of all simplices that have nontrivial intersection with $\sigma$. Note that, if $\Sigma$ is $5$-large, \Cref{lem:describe-combinatorial-sphere} implies that $\hat B(\sigma, \Sigma)$ is a full subcomplex of $\Sigma$. The vertices of $\hat B(\sigma, \Sigma)$ naturally partition into vertices in $\sigma$ and vertices in $\hat S(\sigma, \Sigma)$. Our next lemma describes the topology of the combinatorial $1$-sphere $\hat S(\sigma, \Sigma)$ and the combinatorial $1$-ball $\hat B(\sigma, \Sigma)$.

\begin{figure}[h]
\centering
\begin{tikzpicture}
\newcommand{\BsigmaT}{
    \coordinate (p1) at (0,1);
    \coordinate (p2) at (-1,0);
    \coordinate (p3) at (1,0);
    \coordinate (p4) at ($(p1) + (1.75,0)$);
    \coordinate (p5) at ($(p1) + (0.5,1.5)$);
    \coordinate (p6a) at ($(p1) + (-1.75,0.25)$);
    \coordinate (p6b) at ($(p1) + (-0.75,1.25)$);
    \coordinate (p7) at ($(p2) + (-1,-0.5)$);
    \coordinate (p8) at ($(p2) + (0.75,-1.5)$);
    \coordinate (p9) at ($(p3) + (0,-1)$);
    \coordinate (p10) at ($(p3) + (0.75,-1.25)$);
    \draw[black, cap=round] (p3) -- (p5); \draw[line width=4pt, white, cap=round] (p1)--(p4); \draw[black, cap=round] (p1) -- (p4);
    \draw[black, cap=round] (p1) -- (p2);
    \draw[black, cap=round] (p1) -- (p3);
    \draw[black, cap=round] (p1) -- (p5);
    \draw[black, cap=round] (p1) -- (p6a);
    \draw[black, cap=round] (p1) -- (p6b);
    \draw[black, cap=round] (p2) -- (p3);
    \draw[black, cap=round] (p2) -- (p6a);
    \draw[black, cap=round] (p2) -- (p6a);
    \draw[black, cap=round] (p2) -- (p7);
    \draw[black, cap=round] (p2) -- (p8);
    \draw[black, cap=round] (p3) -- (p8);
    \draw[black, cap=round] (p3) -- (p4);
    \draw[black, cap=round] (p3) -- (p9);
    \draw[black, cap=round] (p3) -- (p10);
    \draw[black, cap=round] (p4) -- (p5);
    \draw[black, cap=round] (p4) -- (p10);
    \draw[black, cap=round] (p5) -- (p6b);
    \draw[black, cap=round] (p6a) -- (p6b);
    \draw[black, cap=round] (p6a) -- (p7);
    \draw[black, cap=round] (p7) -- (p8);
    \draw[black, cap=round] (p8) -- (p9);
    \draw[black, cap=round] (p8) -- (p10);
    \draw[black, cap=round] (p9) -- (p10);
    \draw[ultra thick, black, cap=round] (p1) -- (p2) -- (p3) -- cycle;
    \foreach \pt in {p1,p2,p3}{
        \draw[thick,black,fill=black] (\pt) circle[radius=1.75pt];
    }
    \draw ($0.33*(p1)+0.33*(p2)+0.33*(p3)$) node[text=black] {{\footnotesize $\sigma$}};
    }
\begin{scope}[shift={(-5,0)}]
\BsigmaT
\draw (p5) node[text=black, anchor=south east] {{\footnotesize $\hat{B}(\sigma,\Sigma)$}};
\end{scope}
\begin{scope}[shift={(0,0)}]
\BsigmaT
\newcommand{\eps}{0.4}
\coordinate (e14) at (${1-\eps}*(p1) + {\eps}*(p4)$);
\coordinate (e15) at (${1-\eps}*(p1) + {\eps}*(p5)$);
\coordinate (e16a) at (${\eps}*(p1) + {1-\eps}*(p6a)$);
\coordinate (e16b) at (${\eps}*(p1) + {1-\eps}*(p6b)$);
\coordinate (e26a) at (${\eps}*(p2) + {1-\eps}*(p6a)$);
\coordinate (e27) at (${\eps}*(p2) + {1-\eps}*(p7)$);
\coordinate (e28) at (${\eps}*(p2) + {1-\eps}*(p8)$);
\coordinate (e34) at (${\eps}*(p3) + {1-\eps}*(p4)$);
\coordinate (e35) at ($0.4*(p3) + 0.6*(p5)$);
\coordinate (e38) at (${\eps}*(p3) + {1-\eps}*(p8)$);
\coordinate (e39) at (${\eps}*(p3) + {1-\eps}*(p9)$);
\coordinate (e310) at (${\eps}*(p3) + {1-\eps}*(p10)$);
\draw[very thick, cbred, cap=round] (e15) -- (e16b);
\draw[very thick, cbred, cap=round] (e16b) -- (e16a);
\draw[very thick, cbred, cap=round] (e16a) -- (e26a);
\draw[very thick, cbred, cap=round] (e26a) -- (e27);
\draw[very thick, cbred, cap=round] (e27) -- (e28);
\draw[very thick, cbred, cap=round] (e28) -- (e38);
\fill[fill=cbred, fill opacity=0.25] (e38) -- (e39) -- (e310) -- cycle;
\draw[very thick, cbred, cap=round] (e38) -- (e39) -- (e310);
\draw[very thick, cbred, cap=round] (e38) -- (e310);
\draw[very thick, cbred, cap=round] (e310) -- (e34);
\fill[fill=cbred, fill opacity=0.25] (e34) -- (e35) -- (e15) -- (e14) -- cycle;
\draw[very thick, cbred, cap=round] (e34) -- (e35) -- (e15) -- (e14);
\draw[very thick, cbred, cap=round] (e34) -- (e14);
\foreach \pt in {e14,e15,e16b,e16a,e26a,e27,e28,e34,e35,e38,e39,e310}{
    \draw[thick, cbred, fill=cbred] (\pt) circle[radius=1.75pt];
}
\draw (p5) node[text=cbred, anchor=south east] {{\footnotesize $S(\sigma,\Sigma)$}};
\end{scope}
\begin{scope}[shift={(5,0)}]
\BsigmaT
\draw[ultra thick, cbred, cap=round] (p4) -- (p5) -- (p6b) -- (p6a) -- (p7) -- (p8) -- (p10) -- cycle;
\fill[fill=cbred, fill opacity=0.25] (p8) -- (p9) -- (p10) -- cycle;
\draw[ultra thick, cbred, cap=round] (p8) -- (p9) -- (p10);
\foreach \pt in {p4,p5,p6b,p6a,p7,p8,p9,p10}{
    \draw[thick, cbred, fill=cbred] (\pt) circle[radius=1.75pt];
}
\draw (p5) node[text=cbred, anchor=south east] {{\footnotesize $\hat{S}(\sigma,\Sigma)$}};
\end{scope}
\end{tikzpicture}
\captionof{figure}{The subsets described in \Cref{lem:metric-to-combinatorial-sphere}.}
\end{figure}
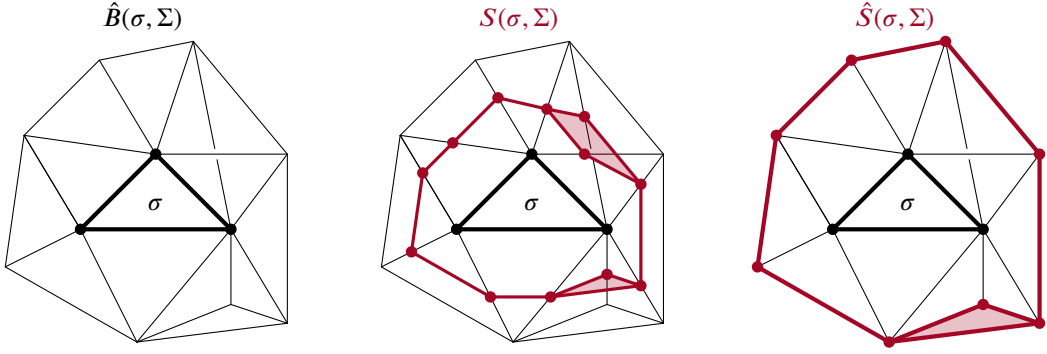

\begin{lem}\label{lem:metric-to-combinatorial-sphere}
    Assume $\Sigma$ is $5$-large, and let $\sigma \subset \Sigma$ be an arbitrary simplex. Then
    \begin{enumerate}
        \item $\hat S(\sigma, \Sigma)$ and $S(\sigma, \Sigma)$ are homotopy equivalent, and
        \item $\hat B(\sigma, \Sigma)$ is contractible.
    \end{enumerate}
\end{lem}

\begin{proof}
We first consider the topology of $\hat S(\sigma, \Sigma)$. There is a natural map $\rho \colon S(\sigma, \Sigma) \rightarrow \hat S(\sigma, \Sigma)$, given by projecting away from the simplex $\sigma$. More explicitly, on each simplex
    \[ \sigma_p*\tau_p = (\sigma_p \times \tau_p \times [0,1])/\sim \]
incident to $\sigma$, the intersection with $S(\sigma, \Sigma)$ is the polyhedral slice $\sigma_p \times \tau_p \times \{\epsilon\}$, while the intersection with $\hat S(\sigma, \Sigma)$ is the top slice $\tau_p \cong (\sigma_p \times \tau_p \times \{1\})/\sim$. The map $\rho$ is then just the projection to $\{1\}$ in the interval coordinate. In terms of the polyhedral structure on $\hat S(\sigma, \Sigma)$, $\rho$ collapses the $\sigma_p$ factor in each polyhedron $\sigma_p \times \tau_p$. Note that $\rho$ is surjective by \Cref{lem:describe-combinatorial-sphere}. Thus the intersection of the preimage under $\rho$ of any point with any given polyhedron of $S(\sigma, \Sigma)$ is contractible or empty.

    
Now let $x\in \hat S(\sigma, \Sigma)$ be an arbitrary point, and let us consider $\rho^{-1}(x)$. There is a unique simplex $\tau_x \subset \hat S(\sigma, \Sigma)$ containing $x$ in its interior. \Cref{lem:describe-combinatorial-sphere} implies that there is a maximal face $\sigma_x \subset \sigma$ associated with $\tau_x$ such $\sigma_x * \tau_x$ is a simplex of $\Sigma$, and $\sigma_x$ is unique since $\Sigma$ is flag. Then from the discussion in the previous paragraph, $\rho^{-1}(x)$ is a copy of $\sigma_x \times \{x\}$ inside the slice $\sigma_x \times \tau_x$, hence it is contractible. A continuous map with contractible point preimages is a cell-like map, and it is well known that a cell-like map between compact spaces homeomorphic to retracts of open subsets of $\mathbb{R}^n$ is a homotopy equivalence (see \cite[Cor.\ 1.3]{Lacher}). This proves that $\rho$ is a homotopy equivalence, as desired.

To verify the corresponding statement for $\hat B(\sigma, \Sigma)$, we note that $\hat B(\sigma, \Sigma)$ is the union of simplices of the form $\sigma_p * \tau_p$, where $\tau_p$ is a face in $\hat S(\sigma, \Sigma)$ and $\sigma_p$ is a face of $\sigma$ by \Cref{lem:describe-combinatorial-sphere}. We now form an auxiliary space $B(\sigma, \Sigma)$ and a continuous quotient map $\phi \colon B(\sigma, \Sigma) \rightarrow \hat B(\sigma, \Sigma)$. Recall that the space $S(\sigma, \Sigma)$ is the $\epsilon$-sphere centered at $\sigma$, identified with the $\epsilon$-slice $\sigma_p \times \tau_p \times \{\epsilon\}$ in the join description of each simplex incident to $\sigma$. We let $B(\sigma, \Sigma)$ denote the closed $\epsilon$-ball centered at $\sigma$. Thus $B(\sigma, \Sigma)$ gets identified with $(\sigma_p \times \tau_p \times [0, \epsilon])/\sim$ in the join description of each simplex incident to $\sigma$; recall that at the $0$-level we collapse the $\tau_p$ coordinates to points. Since we can contract the interval factors down, we clearly have that $B(\sigma, \Sigma)$ deformation retracts to $\sigma$, hence it is contractible. 

The space $B(\sigma, \Sigma)$ contains a copy of $S(\sigma, \Sigma)$ at level $\epsilon$, and we now define $\phi$ to be the quotient map $\rho$ on the subset $S(\sigma, \Sigma)$. As we can stretch out the $[0, \epsilon]$ coordinate to $[0,1]$, we see that the quotient can naturally be identified with $\hat B(\sigma, \Sigma)$. Moreover point preimages under $\phi$ are either points when $x\notin \hat S(\sigma, \Sigma)$ or coincide with the point preimage under $\rho$ when $x\in \hat S(\sigma, \Sigma)$, and hence point preimages are always contractible. Thus the map $\phi$ is cell-like, and hence it is a homotopy equivalence. We conclude that the space $\hat B(\sigma, \Sigma)$ is contractible, as desired.
\end{proof}

\begin{rmk}
An alternative proof is possible for the second statement of \Cref{lem:metric-to-combinatorial-sphere} by noting that the maximal dimensional faces on $\hat S(\sigma, \Sigma)$ are free faces in the complex $\hat B(\sigma, \Sigma)$ by \Cref{lem:describe-combinatorial-sphere}. Thus one can perform simplicial collapses, reducing the maximal dimension of faces in $\hat S(\sigma, \Sigma)$. After finitely many steps, this simplicially collapses $\hat B(\sigma, \Sigma)$ to the complex $\sigma$. 
\end{rmk}

\subsection{Analysis of combinatorial spheres in thickenings}\label{subsec:thickenings}

We now return to analyzing the thickenings $T_n$ from \Cref{sec:MainPf} in order to prove \Cref{prop:cd-spherical-nbhds}. Adopting all notation from that section, recall that the goal is to prove that $\cd(S(\sigma, T_n))\leq n-1$ for all nonempty simplices $\sigma$ of $T_n$. Equivalently, we want to show that $\overline{H}{}^i(S(\sigma, T_n))=0$ for all $i\geq n$ and $\emptyset \neq \sigma \subset T_n$.
    
As a first step, let us describe the topology of links in a general thickening. For any simplex $\sigma \subset \Th_\alpha(X)$, by definition of the $\alpha$-thickening we have that $\alpha(V(\sigma))$ is contained in some cube of $X$. Consider a minimal such cube under containment. In a general cube complex $X$ this minimal cube might not be unique. However we have the following:

\begin{lem}\label{lem:min-cube-unique}
    Let $X$ be a $5$-large cube complex, $\Th_\alpha(X)$ an arbitrary thickening, and $\sigma \subset \Th_\alpha(X)$ a nonempty simplex. Then there exists a unique minimal cube $\square_\sigma \subset X$ containing the vertices $\alpha(V(\sigma))$.
\end{lem}

\begin{proof}
    We are given a set $\alpha(V(\sigma))$ of vertices in the cube complex $X$, which we know is contained in a cube, and want to argue that there is a unique minimal cube containing $\alpha(V(\sigma))$. This is a purely local statement, as it only depends on the combinatorial $1$-neighborhood of a vertex in $\alpha(V(\sigma))$. From the $5$-large hypothesis, such a neighborhood lifts to $\tilde X$. Therefore it suffices to consider this question in $\tilde X$, which is a CAT($0$) cube complex. However any two cubes in a CAT($0$) cube complex can only intersect in a common subcube. Now assume $\square$ is another minimal cube containing the set $\alpha(V(\sigma))$. Then the intersection $\square \cap \square_\sigma$ is a subcube of both $\square_\sigma$, and $\square$. Minimality then implies that $\square_\sigma = \square \cap \square_\sigma = \square$, which implies the minimal cube is unique. 
\end{proof}

Notice that in our setting, the cube complexes we are thickening are always $5$-large, therefore \Cref{lem:min-cube-unique} applies. We will denote by $\square_\sigma$ the unique minimal cube containing the set $\alpha(V(\sigma))$. If we have a proper containment of simplices $\sigma \subsetneq \tau$, then clearly $\square _\sigma \subseteq \square_\tau$. Note that the containment of the corresponding cubes might not be proper, i.e., it could be that $\square _\sigma =\square_\tau$ despite the fact that $\sigma \neq \tau$. Our next lemma will exploit this feature and use it to relate the topology of the link of $\sigma$ with the topology of the link of $\square_\sigma$.

\begin{lem}\label{lem:three-props}
Given any simplex $\sigma \subset \Th_\alpha(X)$, the link $\lk(\sigma, \Th_\alpha(X))$ contains a pair of subcomplexes $\lk^\prime(\sigma, \Th_\alpha(X))$, $\lk^{\prime \prime}(\sigma, \Th_\alpha(X))$ with the following three properties:
\begin{enumerate}
\item $\lk(\sigma, \Th_\alpha(X))$ decomposes as a join $\lk^\prime(\sigma, \Th_\alpha(X))*\lk^{\prime \prime}(\sigma, \Th_\alpha(X))$;
\item if $\lk^\prime(\sigma, \Th_\alpha(X))$ is nonempty, then it is contractible;
\item the complex $\lk^{\prime \prime}(\sigma, \Th_\alpha(X))$ is homotopy equivalent to $\lk(\square_\sigma , X)$.
\end{enumerate}
\end{lem}

\begin{proof}
We start by defining the subcomplex $\lk^\prime(\sigma, \Th_\alpha(X))$ (see \Cref{fig:vcd-arg1}) and verifying property (2). Consider the vertices in the set $\alpha^{-1}(V(\square_\sigma))$. From the definition of the thickening $\Th_\alpha(X)$, these span a simplex containing $\sigma$. Thus all vertices in this set are either at combinatorial distance zero or one from $\sigma$, according to whether they lie in $\sigma$ or not. We let $\lk^\prime(\sigma, \Th_\alpha(X))$ be spanned by those vertices at distance one. This subcomplex is either empty or consists of a single simplex, and hence it is contractible. This establishes property (2).

\begin{figure}[h!]
    \centering
\begin{tikzpicture}
    \newcommand{\scalar}{0.75}
    \newcommand{\lrx}{(\scalar)*3} 
    \newcommand{\lry}{(\scalar)*0} 
    \newcommand{\fbx}{(\scalar)*1.5} 
    \newcommand{\fby}{(\scalar)*1.5} 
    \coordinate (cube1) at (0,0);
    \coordinate (cube2) at ($(cube1)+({\lrx},{\lry})$);
    \coordinate (cube3) at ($(cube2)+({\fbx},{\fby})$);
    \coordinate (cube4) at ($(cube3)-({\lrx},{\lry})$);
    \draw[thick, black] (cube1) -- (cube2) -- (cube3) -- (cube4) -- cycle;
    \foreach \pt in {cube1,cube3}{
        \draw[thick,cbred,fill=cbred] (\pt) circle[radius=2.5pt];
    }
    \foreach \pt in {cube2,cube4}{
        \draw[thick,black,fill=black] (\pt) circle[radius=2.5pt];
    }
    \coordinate (v1a) at ($(cube1)+(0,{\scalar*3})$);
    \coordinate (v1b) at ($(cube1)+(0,{\scalar*4})$);
    \coordinate (v1c) at ($(cube1)+(0,{\scalar*5})$);
    \coordinate (v2a) at ($(cube2)+(0,{\scalar*3})$);
    \coordinate (v2b) at ($(cube2)+(0,{\scalar*4})$);
    \coordinate (v2c) at ($(cube2)+(0,{\scalar*5})$);
    \coordinate (v3a) at ($(cube3)+(0,{\scalar*3})$);
    \coordinate (v3b) at ($(cube3)+(0,{\scalar*4})$);
    \coordinate (v3c) at ($(cube3)+(0,{\scalar*5})$);
    \coordinate (v4a) at ($(cube4)+(0,{\scalar*3})$);
    \coordinate (v4b) at ($(cube4)+(0,{\scalar*4})$);
    \coordinate (v4c) at ($(cube4)+(0,{\scalar*5})$);
    \foreach \pt in {v1a,v2a,v3a,v4a}{
        \draw[ultra thick,black, dashed,-latex] (\pt) -- ($(\pt) + (0,{\scalar*(-2)})$);
    }
    \foreach \center in {v1b,v2b,v3b,v4b}{
        \draw[very thick,black, opacity=0.3] (\center) ellipse ({\scalar*0.5} and {\scalar*1.5});
    }
    \foreach \x in {v1b,v1c,v2a,v2b,v2c,v3a,v3b,v4a,v4b,v4c}{
        \foreach \y in {v1b,v1c,v2a,v2b,v2c,v3a,v3b,v4a,v4b,v4c}{
            \draw[thin,cbgreen,opacity=0.75] (\x) -- (\y);
        }}
    \draw[line width=5pt,white] (v1a) -- (v3c);
    \draw[thick,cbred] (v1a) -- (v3c);
    \draw[thick,cbred,fill=cbred] (v1a) circle[radius=2pt] node[anchor=south,text=cbred] {{\footnotesize $\sigma$}};
    \draw[thick,cbred,fill=cbred] (v4a) circle[radius=2pt];
    \foreach \pt in {v1b,v1c,v2a,v2b,v2c,v3a,v3b,v4a,v4b,v4c}{
        \draw[thick,cbgreen,fill=cbgreen] (\pt) circle[radius=2pt];
    }
    \foreach \pt in {v1a,v3c}{
        \draw[thick,cbred,fill=cbred] (\pt) circle[radius=2.5pt];
    }
    \node[anchor=west,color=cbred] (alpha) at (4,0) {{\footnotesize $\alpha(V(\sigma))$}};
    \draw[cbred,densely dotted,shorten >= 0.2cm,-latex] (alpha) to[out=150,in=20] (cube1);
    \draw[cbred,densely dotted,shorten >= 0.2cm,-latex] (alpha) -- (cube3);
    \draw (1.55,0) node[anchor=north,color=black] {{\footnotesize$\square_{\sigma}$}};
    \draw ($(v1a)+(0,-1)$) node[anchor=east,text=black] {{\footnotesize$\alpha$ map}};
    \draw ({\scalar*5},{\scalar*5}) node[anchor=west,text=cbgreen] {{\footnotesize \begin{tabular}{c}Maximal simplex\\in $\lk'(\sigma,\Th_{\alpha}(X))$\end{tabular}}};
\end{tikzpicture}

    \captionsetup{margin=0.75in}
    \captionof{figure}{Given a simplex $\sigma$ (red) in $\Th_\alpha(X)$, let $\square_\sigma$ (black) be the minimal cube in $X$ containing $\alpha\!\left(V(\sigma)\right)$. The maximal simplex (green) in $\lk^{\prime}(\sigma,\Th_\alpha(X))$ is spanned by the vertices in $\alpha^{-1}\!\left(V\!\left(\square_\sigma\right)\right) \setminus V(\sigma)$.}
    \label{fig:vcd-arg1}
\end{figure}
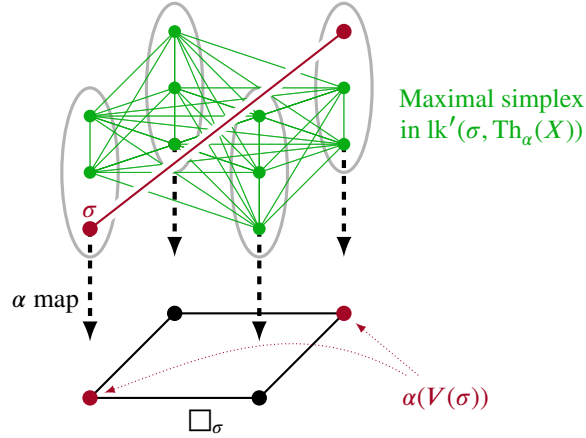

Next let us define the complex $\lk^{\prime \prime}(\sigma, \Th_\alpha(X))$ and establish property (1). This is the subcomplex of $\lk(\sigma, \Th_\alpha(X))$ spanned by vertices $v$ with the property that $\alpha(v)\notin \square_\sigma$. In particular, the vertices of $\lk(\sigma, \Th_\alpha(X))$ are partitioned between $\lk^\prime(\sigma, \Th_\alpha(X))$ and $\lk^{\prime \prime}(\sigma, \Th_\alpha(X))$ according to whether $\alpha(v)\in \square_\sigma$ or $\alpha(v)\notin \square_\sigma$. This implies the subcomplexes $\lk^\prime(\sigma, \Th_\alpha(X))$ and $\lk^{\prime \prime}(\sigma, \Th_\alpha(X))$ are disjoint subcomplexes inside $\lk(\sigma, \Th_\alpha(X))$ and gives the containment 
    \[ \lk(\sigma, \Th_\alpha(X)) \subseteq \lk^\prime(\sigma, \Th_\alpha(X))*\lk^{\prime \prime}(\sigma, \Th_\alpha(X)). \]
For the reverse containment, given a pair of simplices $\tau^\prime \subseteq \lk^\prime(\sigma, \Th_\alpha(X))$ and $\tau^{\prime \prime} \subseteq \lk^{\prime \prime}(\sigma, \Th_\alpha(X))$ we have that $\alpha(V(\tau^\prime))\subseteq \square_\sigma$, while $\alpha(V(\tau^{\prime \prime}))\subseteq \square$ for some cube $\square_\sigma \subsetneq \square$. It follows that $\alpha(V(\tau^\prime) \cup V(\tau^{\prime \prime})) \subseteq \square$ and hence that $\tau^\prime * \tau^{\prime \prime}$ defines a simplex in $\lk(\sigma, \Th_\alpha(X))$. This yields the reverse containment and establishes property (1).

\medskip

To complete the proof of the lemma, it remains to verify property (3). Let $Z$ denote the union of all the cubes containing $\square_\sigma$ and $\hat Z$ be the cube subcomplex of $Z$ consisting of all cubes in $Z$ that do not intersect $\square_\sigma$. There is a natural homeomorphism $\hat Z \cong \square_\sigma \times \lk(\square_\sigma , X)$, so $\hat Z$ is homotopy equivalent to $\lk(\square_\sigma , X)$. From the definition of $\lk^{\prime \prime}(\sigma, \Th_\alpha(X))$, any simplex $\tau\subset \lk^{\prime \prime}(\sigma, \Th_\alpha(X))$ has the property that $\alpha(V(\tau))$ is disjoint from $\square_\sigma$, and hence it lies in $V(\hat Z)$. In other words, the entire complex $\lk^{\prime \prime}(\sigma, \Th_\alpha(X))$ lies ``above'' the cube complex $\hat Z$. 

Now consider the thickening $\Th_\alpha(\hat Z)$ obtained by the restriction $\alpha \colon \alpha^{-1}(V(\hat Z)) \rightarrow V( \hat Z)$. This is a subcomplex of $\lk^{\prime \prime}(\sigma, \Th_\alpha(X))$ that is homotopy equivalent to $\hat Z$, and hence to $\lk(\square_\sigma , X)$. To complete the verification of property (3) it therefore suffices to show that $\lk^{\prime \prime}(\sigma, \Th_\alpha(X))$ deformation retracts to $\Th_\alpha(\hat Z)$. 

\begin{figure}[h!]
    \centering
\begin{tikzpicture}
    \newcommand{\scalar}{0.75}
    \newcommand{\lrx}{(\scalar)*3} 
    \newcommand{\lry}{(\scalar)*0} 
    \newcommand{\fbx}{(\scalar)*1.5} 
    \newcommand{\fby}{(\scalar)*1.5} 
    \coordinate (cube1) at (0,0);
    \coordinate (cube2) at ($(cube1)+({\lrx},{\lry})$);
    \coordinate (cube3) at ($(cube2)+({\fbx},{\fby})$);
    \coordinate (cube4) at ($(cube3)-({\lrx},{\lry})$);
    \draw[white,pattern=north west lines, pattern color=cbblue, opacity=0.75] (cube1) -- (cube2) -- (cube3) -- (cube4) -- cycle;
    \draw[thick, black] (cube1) -- (cube2) -- (cube3) -- (cube4);
    \foreach \pt in {cube1,cube4}{
        \draw[thick,cbred,fill=cbred] (\pt) circle[radius=2.5pt];
    }
    \foreach \pt in {cube2,cube3}{
        \draw[thick,black,fill=black] (\pt) circle[radius=2.5pt];
    }
    \draw[thick,cbred] (cube4) to node[yshift=1.1em] {{\footnotesize $\square_\sigma$}}  (cube1);
    \coordinate (v1a) at ($(cube1)+(0,{\scalar*3})$);
    \coordinate (v1b) at ($(cube1)+(0,{\scalar*4})$);
    \coordinate (v1c) at ($(cube1)+(0,{\scalar*5})$);
    \coordinate (v2a) at ($(cube2)+(0,{\scalar*3})$);
    \coordinate (v2b) at ($(cube2)+(0,{\scalar*4})$);
    \coordinate (v2c) at ($(cube2)+(0,{\scalar*5})$);
    \coordinate (v3a) at ($(cube3)+(0,{\scalar*3})$);
    \coordinate (v3b) at ($(cube3)+(0,{\scalar*4})$);
    \coordinate (v3c) at ($(cube3)+(0,{\scalar*5})$);
    \coordinate (v4a) at ($(cube4)+(0,{\scalar*3})$);
    \coordinate (v4b) at ($(cube4)+(0,{\scalar*4})$);
    \coordinate (v4c) at ($(cube4)+(0,{\scalar*5})$);
    \foreach \pt in {v1a,v2a,v3a,v4a}{
        \draw[ultra thick,black, dashed,-latex] (\pt) -- ($(\pt) + (0,{\scalar*(-2)})$);
    }    
    \foreach \x in {v1b,v1c,v4b,v4c}{
        \foreach \y in {v2a,v2b,v2c,v3a,v3b,v3c}{
            \draw[very thin,gray,opacity=0.6] (\x) -- (\y);
        }
    }
    \draw[thick,cbred] (v4a) -- (v1a);
    \draw[thick,cbred,fill=cbred] (v1a) circle[radius=2pt] node[anchor=south,text=cbred] {{\footnotesize $\sigma$}};
    \draw[thick,cbred,fill=cbred] (v4a) circle[radius=2pt];
    \foreach \x in {v1b,v1c,v4b,v4c}{
        \draw[thick,cbgreen,fill=cbgreen] (\x) circle[radius=2pt];
        \foreach \y in {v1b,v1c,v4b,v4c}{
            \draw[cbgreen,opacity=0.75] (\x) -- (\y);
        }}
    \foreach \x in {v2a,v2b,v2c,v3a,v3b,v3c}{
        \draw[thick,cbblue,fill=cbblue] (\x) circle[radius=2pt];
        \foreach \y in {v2a,v2b,v2c,v3a,v3b,v3c}{
            \draw[cbblue,opacity=0.75] (\x) -- (\y);
        }
    }
    \foreach \center in {v1b,v2b,v3b,v4b}{
        \draw[thick,black, opacity=0.3] (\center) ellipse ({\scalar*0.5} and {\scalar*1.5});
    }
    \draw ($(v3a)+(0,-1)$) node[anchor=west,text=black] {{\footnotesize $\alpha$ map}};
    \draw ($0.6*(cube2)+0.4*(cube3)$) node[anchor=west,text=cbblue] {{\footnotesize \begin{tabular}{c}maximal\\square $\square$\\containing $\square_\sigma$\end{tabular}}};
    \draw ($(v4c)+(-0.5,0)$) node[anchor=east,text=cbgreen] {{\footnotesize $\mathrm{lk}'(\sigma,\Th_{\alpha}(X))$}};
    \draw ($(v3b)+(0.5,0)$) node[anchor=west,text=cbblue] {{\footnotesize\begin{tabular}{c}simplex in \\ $\lk^{\prime \prime}(\sigma,\Th_{\alpha}(X))$ \\ lying above $\square$ \end{tabular}}};
\end{tikzpicture}
    \captionsetup{margin=0.75in}
    \captionof{figure}{A general simplex in $\lk(\sigma,\Th_\alpha(X))$ splits as a join (gray) of simplices in $\mathrm{lk}^\prime(\sigma,\Th_\alpha(X))$ (green) and $\lk^{\prime \prime}(\sigma,\Th_{\alpha}(X))$ (blue)}
    \label{fig:vcd-arg2}
\end{figure}
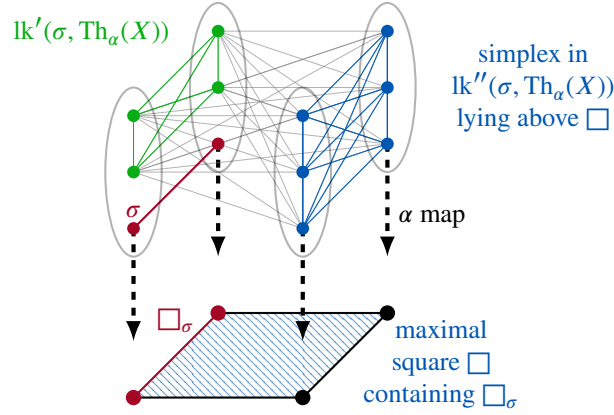

Recall that $\hat Z$ is defined as a subcomplex of the cube complex $Z$. To construct the deformation retraction, we proceed by induction on the dimension of the cubes in $Z$. More precisely, we will consider a cube $\square \subset Z$ containing $\square_\sigma$ and the restriction of the two complexes $\lk^{\prime \prime}(\sigma, \Th_\alpha(X))$ and $\Th_\alpha(\hat Z)$ lying above $\hat Z \cap \square$, i.e., whose vertices $\alpha$-project to vertices in $\hat Z \cap \square$. We induct on $\dim(\square)$ and show that in each dimension we can choose a deformation retraction compatible with the previously-constructed deformation retractions. This will allow the deformation retractions associated with each cube to glue compatibly and yield a global deformation retraction of $\lk^{\prime \prime}(\sigma, \Th_\alpha(X))$ to $\Th_\alpha(\hat Z)$.

\begin{figure}[h!]
    \centering
\begin{tikzpicture}
\newcommand{\scalar}{0.75}
\newcommand{\lrx}{(\scalar)*3} 
\newcommand{\lry}{(\scalar)*0} 
\newcommand{\fbx}{(\scalar)*1.5} 
\newcommand{\fby}{(\scalar)*1.5} 
\newcommand{\udx}{(\scalar)*0}
\newcommand{\udy}{(\scalar)*2.5}
\newcommand{\cubecoords}{ 
    \coordinate (cube1) at (0,0);
    \coordinate (cube2) at ($(cube1)+({\lrx},{\lry})$);
    \coordinate (cube3) at ($(cube2)+({\fbx},{\fby})$);
    \coordinate (cube4) at ($(cube3)-({\lrx},{\lry})$);
    \coordinate (cube5) at ($(cube1)+({\udx},{\udy})$);
    \coordinate (cube6) at ($(cube2)+({\udx},{\udy})$);
    \coordinate (cube7) at ($(cube3)+({\udx},{\udy})$);
    \coordinate (cube8) at ($(cube4)+({\udx},{\udy})$);
}
\draw[ultra thick,dashed,-latex] ({\scalar*(-1.5)},{\scalar*(-0.25)}) to node[anchor=east] {{\footnotesize $\alpha$ map}} ({\scalar*(-1.5)},{\scalar*(-2.25)});
\draw[ultra thick,dashed,cbblue,-latex] ({\scalar*(0)},{\scalar*2}) to node[anchor=north,text=cbblue] {{\footnotesize deformation retract}} ({\scalar*2},{\scalar*2});
\draw[ultra thick,dashed,-latex] ({\scalar*7},{\scalar*(-0.25)}) to node[anchor=east] {{\footnotesize $\alpha$ map}} ({\scalar*7},{\scalar*(-2.25)});
\draw[ultra thick,dashed,cbblue,-latex] ({\scalar*(2)},{\scalar*(-5)}) to node[anchor=north,text=cbblue] {{\footnotesize radial projection}} ({\scalar*0},{\scalar*(-5)});
\begin{scope}[shift={({\scalar*(-5.5)},{\scalar*(-7)})}]
    \cubecoords
    \coordinate (link1a) at ($(cube1)+({\scalar*0.75},0)$);
    \coordinate (link1b) at ($(cube1)+(0,{\scalar*0.75})$);
    \coordinate (link2a) at ($(cube4)+({\scalar*0.75},0)$);
    \coordinate (link2b) at ($(cube4)+(0,{\scalar*0.75})$);
    \coordinate (midpt1) at ($0.5*(link1a)+0.5*(link2a)$);
    \coordinate (midpt2) at ($0.5*(link1b)+0.5*(link2b)$);
    \coordinate (midsigma) at ($0.5*(cube1)+0.5*(cube4)$);
    \draw[cbred] (cube1) -- (cube4);
    \draw[very thick,cbblue] (midpt1) arc (0:90:{\scalar*0.75});
    \draw[opacity=0,cbblue, shade, top color=white, bottom color=cbblue, shading angle=45, fill opacity=0.5] (link1a) arc(0:90:{\scalar*0.75}) to (link2b) arc(90:0:{\scalar*0.75}) to (link1a);
    \draw[thick,black] (cube3) -- (cube4) -- (cube8) -- (cube7) -- cycle;
    \draw[thick,black] (cube1) -- (cube2) -- (cube3);
    \draw[thick,black] (cube1) -- (cube5) -- (cube8);
    \draw[line width=4pt,white] (cube5) -- (cube6) -- (cube7);
    \draw[thick,black] (cube5) -- (cube6) -- (cube7);
    \draw[line width=4pt,white] (cube2) -- (cube6);
    \draw[thick,black] (cube2) -- (cube6);
    \foreach \pt in {cube1,cube4}{
        \draw[thick,cbred,fill=cbred] (\pt) circle[radius=2pt];
    }
    \foreach \pt in {cube2,cube3,cube5,cube6,cube7,cube8}{
        \draw[thick,black,fill=black] (\pt) circle[radius=2pt];
    }
    \draw (cube1) node[anchor=south east,text=cbred] {$\square_\sigma$};
    \draw ($0.5*(cube1)+0.5*(cube2)$) node[anchor=north,text=black] {{\footnotesize cube $\square$ in $Z$}};
    \node[text=cbblue] (words) at ({\scalar*2},{\scalar*(-1.5)}) {{\footnotesize \begin{tabular}{c}portion of\\$\mathrm{lk}(\square_\sigma,X) \times \square_\sigma$\\in the cube $\square$\end{tabular}}};
    \draw[cbblue,densely dashed,-latex] (words) to[out=30,in=-60] ({\scalar*4.5},{\scalar*0.5}) to[out=120,in=0] ($0.8*(link1a)+0.2*(link2a)+(0.25,0)$);
    \foreach \ang in {0,15,30,...,90}{
        \draw[cbblue,-latex] ($(midsigma)+({\scalar*0.75*cos(\ang)},{\scalar*0.75*sin(\ang)})$) -- ($(midsigma)+({cos(\ang)},{sin(\ang)})$);
    }
\end{scope}
\begin{scope}[shift={({\scalar*(-5.5)},{\scalar*0})}]
    \cubecoords
    \draw[thick,cbred] (cube1) -- (cube4);
    \foreach \x in {cube2,cube3,cube5,cube6,cube7,cube8}{
        \foreach \y in {cube2,cube3,cube5,cube6,cube7,cube8}{
        \draw[line width=4pt,white] (\x) -- (\y);
        \draw[thick,cbblue] (\x) -- (\y);
      }
    }
    \foreach \pt in {cube1,cube4}{
        \draw[thick,cbred,fill=cbred] (\pt) circle[radius=2pt];
    }
    \foreach \pt in {cube2,cube3,cube5,cube6,cube7,cube8}{
        \draw[thick,cbblue,fill=cbblue] (\pt) circle[radius=2pt];
    }
    \draw[thick,cbblue,fill=cbblue,fill opacity=0.15] (cube5) -- (cube6) -- (cube2) -- (cube3) -- (cube7) -- (cube8) -- cycle;
    \draw[thick,cbblue,fill=cbblue!75!black, fill opacity=0.25] (cube5) -- (cube6) -- (cube2) -- cycle;
    \draw ($0.5*(cube1)+0.5*(cube4)$) node[anchor=south east,text=cbred] {{\footnotesize $\sigma$}};
    \draw ($0.5*(cube8)+0.5*(cube7)$) node[anchor=south,text=cbblue] {{\footnotesize \begin{tabular}{c}simplex in $\lk^{\prime \prime}(\sigma,\Th(X)$\\ lying above $\square$\end{tabular}}};
\end{scope}
\begin{scope}[shift={({\scalar*3},{\scalar*0})}]
    \cubecoords
    \draw[thick,cbred] (cube1) -- (cube4);
    \foreach \x in {cube2,cube3,cube6,cube7}{
        \foreach \y in {cube2,cube3,cube6,cube7}{
        \draw[line width=4pt,white] (\x) -- (\y);
        \draw[thick,cbblue] (\x) -- (\y);
      }
    }
    \foreach \x in {cube5,cube6,cube7,cube8}{
        \foreach \y in {cube5,cube6,cube7,cube8}{
        \draw[line width=4pt,white] (\x) -- (\y);
        \draw[thick,cbblue] (\x) -- (\y);
      }
    }
    \foreach \pt in {cube1,cube4}{
        \draw[thick,cbred,fill=cbred] (\pt) circle[radius=2pt];
    }
    \foreach \pt in {cube2,cube3,cube5,cube6,cube7,cube8}{
        \draw[thick,cbblue,fill=cbblue] (\pt) circle[radius=2pt];
    }
    \draw[thick,cbblue,fill=cbblue, fill opacity=0.1] (cube5) -- (cube6) -- (cube2) -- (cube3) -- (cube7) -- (cube8) -- cycle;
    \draw ($0.5*(cube1)+0.5*(cube4)$) node[anchor=south east,text=cbred] {{\footnotesize $\sigma$}};
    \draw ($0.5*(cube7)+0.5*(cube8)$) node[anchor=south,text=cbblue] {{\footnotesize $\Th_{\alpha}(\hat{Z} \cap \square)$}};
\end{scope}
\begin{scope}[shift={({\scalar*3},{\scalar*(-7)})}]
    \cubecoords
    \draw[cbred] (cube1) -- (cube4);
    \draw[thick,black] (cube2) -- (cube1) -- (cube5);
    \draw[thick,black] (cube3) -- (cube4) -- (cube8);
    \draw[line width=4pt,white] (cube5) -- (cube6) -- (cube7) -- (cube8) -- cycle;
    \draw[line width=4pt,white] (cube2) -- (cube3) -- (cube7) -- (cube6) -- cycle;
    \draw[thick,cbblue,fill=cbblue,fill opacity=0.1] (cube5) -- (cube6) -- (cube7) -- (cube8) -- cycle;
    \draw[thick,cbblue,fill=cbblue,fill opacity=0.1] (cube2) -- (cube3) -- (cube7) -- (cube6) -- cycle;
    \foreach \pt in {cube1,cube4}{
        \draw[thick,cbred,fill=cbred] (\pt) circle[radius=2pt];
    }
    \foreach \pt in {cube2,cube3,cube5,cube6,cube7,cube8}{
        \draw[thick,cbblue,fill=cbblue] (\pt) circle[radius=2pt];
    }
    \draw (cube1) node[anchor=south east,text=cbred] {{\footnotesize $\square_\sigma$}};
    \draw ($0.5*(cube8)+0.5*(cube7)$) node[anchor=south,text=cbblue] {{\footnotesize $\hat{Z} \cap \square$}};
    \draw ($0.5*(cube1)+0.5*(cube2)$) node[anchor=north,text=black] {{\footnotesize cube $\square$ in $Z$}};
\end{scope}
\end{tikzpicture}
    \captionsetup{margin=0.75in}
    \captionof{figure}{Sequence of homotopy equivalences illustrating the argument for Property (3) in \Cref{lem:three-props}}
    \label{fig:vcd-arg3}
\end{figure}
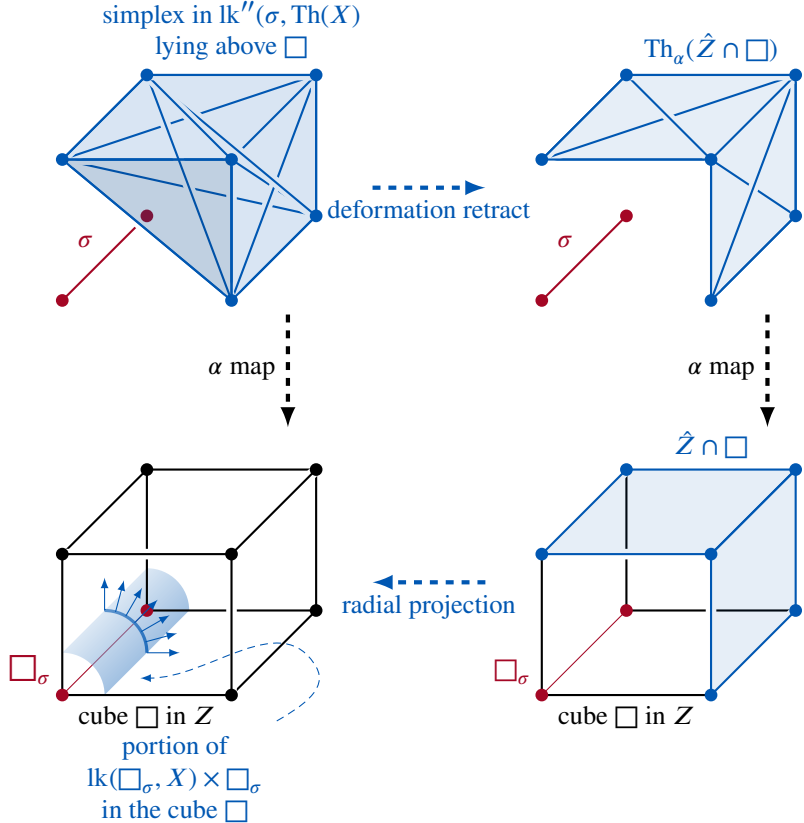

The bottom of the induction is the case of a cube $\square \subset Z$ with $\dim(\square) = \dim(\square_\sigma) +1$. In this case, $\square = \square_\sigma \times I$, and $\hat Z \cap \square$ is the face of $\square$ opposite $\square_\sigma$. This forces the two subcomplexes above $\hat Z \cap \square$ to coincide, and hence we can take the identity map as the deformation retraction. Now for the inductive step, we assume that we are already provided the deformation retractions for the portions of the simplex that lie above the relevant lower dimensional faces of the cube $\hat Z \cap \square$. We would like to extend the deformation retraction to the rest of the simplicial complex lying above $\hat Z \cap \square$.

In general, if we have a simplicial complex $A$ and a subcomplex $B\subseteq A$, the obstruction to deformation retracting $A$ to $B$ lies in the relative homotopy groups $\pi_*(A, B)$. In our setting, we already have a deformation retraction from some intermediate subspace $B\subseteq B^\prime \subseteq A$, and would like to extend the deformation retraction to one from $A$ to $B$. Again, the obstructions are the relative homotopy groups $\pi_*(A, B^\prime) \cong \pi_*(A,B)$. If these vanish, then one can construct the desired deformation retraction by composing a deformation retraction from $A$ to $B^\prime$ with the given deformation retraction from $B^\prime$ to $B$. Thus the problem reduces to understanding the topology of the two sets we are studying.

To this end, let us identify the topology of the part of $\lk^{\prime \prime}(\sigma, \Th_\alpha(X))$ lying above $\hat Z \cap \square$. From the definition, the vertices of this subcomplex are precisely those whose $\alpha$-projection are contained in $V(\square) \setminus V(\square_\sigma)$. Since these vertices have image contained in a single cube, they span a simplex. Thus the restriction of the complex $\lk^{\prime \prime}(\sigma, \Th_\alpha(X))$ to the preimage under $\a$ of the cube $\square$ is a single simplex $\Delta$, and thus it is a contractible set. 

On the other hand, the subcomplex $\Th_\alpha(\hat Z \cap \square)$ is the $\alpha$-thickening of the cube complex $\hat Z \cap \square$. However this cube complex is spanned by the vertices $V(\square) \setminus V(\square_\sigma)$, and thus is homotopic to the contractible set $\partial \square \setminus \square_\sigma$. We conclude that $\Th_\alpha(\hat Z \cap \square)$ is also a contractible subset of $\Delta$. Since both spaces are contractible, the relative homotopy groups all vanish and the desired deformation retraction exists. This completes the inductive step and verifies statement (3), completing the proof of the lemma.
\end{proof}

Informally, \Cref{lem:three-props} shows that, up to homotopy, the links of simplices in a thickening are either trivial or coincide with the links in the underlying cube complex. Note the following immediate consequence of the previous result.

\begin{lem}\label{lem:cdim-of-links}
    If $\sigma$ is an arbitrary nonempty simplex in $T_n$, then $\cdim(\lk(\sigma, T_n)) \leq n-1$.
\end{lem}

\begin{proof}
Since $T_n$ is a thickening, we can apply \Cref{lem:three-props} to see that $\lk(\sigma, T_n)$ decomposes as a join $\lk^\prime(\sigma, T_n)*\lk^{\prime\prime}(\sigma, T_n)$. If $\lk^\prime(\sigma, T_n)\neq \emptyset$, then $\lk^\prime(\sigma, T_n)$ is a nonempty simplex, hence it is contractible. This forces $\lk(\sigma, T_n)$ to likewise be contractible, in which case we are done.

Alternatively, if $\lk^\prime(\sigma, T_n)= \emptyset$, then \Cref{lem:three-props} tells us that
    \[ \lk(\sigma, T_n) = \lk^{\prime\prime}(\sigma, T_n) \simeq \lk(\square_\sigma, X_n). \]
We let $v\in \square_\sigma$ be an arbitrary vertex of the cube $\square_\sigma$, and apply \Cref{lem:links-cube-to-vertex} to obtain a combinatorial isomorphism
    \[ \lk(\square_\sigma, X_n) \cong \lk(\tau, \lk(v,X)) \cong \lk(\tau, T_{n-1}) \]
where $\tau$ is the (possibly empty) simplex in $\lk(v,X)$ corresponding to $\square_\sigma$.

If the simplex $\tau$ is empty, this means $\square_\sigma$ coincides with the vertex $v$ and $\lk(\tau, T_{n-1}) = T_{n-1}$. In that case we conclude by induction with respect to proving the Main Theorem that
    \[ \cdim(\lk(\sigma, T_n)) = \cdim(T_{n-1}) = n-1. \]
On the other hand, if $\tau$ is nonempty then we again obtain
    \[ \cdim(\lk(\sigma, T_n)) = \cdim(\lk(\tau,T_{n-1})) \leq n-2 \]
by induction with respect to both this lemma and the Main Theorem. Either way we obtain the desired bound, which concludes the proof of the lemma.
\end{proof}

In order to compute the cohomological dimension of the polyhedral complex $S(\sigma, T_n)$, we will instead compute the cohomological dimension of the simplicial complex $\hat S(\sigma, T_n)$. Our next step is to give a partition of vertices in $\hat S(\sigma, T_n)$, according to their closest face of $\sigma$.

\medskip

Note that every vertex $v$ in $\hat S(\sigma, T_n)$ is at distance one from some vertex of $\sigma$ by definition. For such a vertex $v$, we can look at \textbf{all} the vertices of $\sigma$ that are adjacent to $v$. This defines a face of $\sigma$, and gives us a map $\pi \colon V(\hat S(\sigma, T_n)) \rightarrow \mathcal F(\sigma)$ from the set of vertices in $\hat S(\sigma, T_n)$ to the set $\mathcal F(\sigma)$ of faces of $\sigma$. We think of $\pi(v)$ as encoding the face of $\sigma$ that is closest to the vertex $v$. We can now partition the vertices of $\hat S(\sigma, T_n)$ according to their image under $\pi$; see \Cref{fig:combinatorial-filtration}. Next, use this partition to form a filtration of $\hat S(\sigma, T_n)$ by subcomplexes. Denote by $S_i$ the subcomplex of $\hat S(\sigma, T_n)$ spanned by the preimage under $\pi$ of all faces of codimension $\leq i$. We clearly have containments $S_0 \subset S_1 \subset \cdots \subset S_k = \hat S(\sigma, T_n)$, where $k=\dim(\sigma)$. This gives us a filtration of the complex $\hat S(\sigma, T_n)$ by the subcomplexes $S_i$; again, see \Cref{fig:combinatorial-filtration}.

\begin{figure}[h]
\centering
\begin{tikzpicture}
\newcommand{\BsigmaT}{
    \coordinate (p1) at (0,1);
    \coordinate (p2) at (-1,0);
    \coordinate (p3) at (1,0);
    \coordinate (p4) at ($(p1) + (1.75,0)$);
    \coordinate (p5) at ($(p1) + (0.5,1.5)$);
    \coordinate (p6a) at ($(p1) + (-1.75,0.25)$);
    \coordinate (p6b) at ($(p1) + (-0.75,1.25)$);
    \coordinate (p7) at ($(p2) + (-1,-0.5)$);
    \coordinate (p8) at ($(p2) + (0.75,-1.5)$);
    \coordinate (p9) at ($(p3) + (0,-1)$);
    \coordinate (p10) at ($(p3) + (0.75,-1.25)$);
    \draw[gray, densely dashed, cap=round] (p3) -- (p5); \draw[line width=4pt, white, cap=round] (p1)--(p4); \draw[gray, densely dashed, cap=round] (p1) -- (p4);
    \draw[gray, densely dashed, cap=round] (p1) -- (p5);
    \draw[gray, densely dashed, cap=round] (p1) -- (p6a);
    \draw[gray, densely dashed, cap=round] (p1) -- (p6b);
    \draw[gray, densely dashed, cap=round] (p2) -- (p6a);
    \draw[gray, densely dashed, cap=round] (p2) -- (p7);
    \draw[gray, densely dashed, cap=round] (p2) -- (p8);
    \draw[gray, densely dashed, cap=round] (p3) -- (p8);
    \draw[gray, densely dashed, cap=round] (p3) -- (p4);
    \draw[gray, densely dashed, cap=round] (p3) -- (p9);
    \draw[gray, densely dashed, cap=round] (p3) -- (p10);
    \draw[gray, densely dashed, cap=round] (p4) -- (p10);
    \draw[gray, densely dashed, cap=round] (p8) -- (p9);
    \draw[ultra thick, gray, cap=round] (p1) -- (p2) -- (p3) -- cycle;
    \foreach \pt in {p1,p2,p3}{
        \draw[thick,gray,fill=gray] (\pt) circle[radius=1.75pt];
    }
    \draw ($0.33*(p1)+0.33*(p2)+0.33*(p3)$) node[text=black] {{\footnotesize $\sigma$}};
    \draw[thick, densely dashed, gray, cap=round] (p4) -- (p5) -- (p6b) -- (p6a) -- (p7) -- (p8) -- (p10) -- cycle;
    \draw[thick, densely dashed, gray, cap=round] (p8) -- (p9) -- (p10);
    \draw[thick, densely dashed, gray, cap=round] (p8) -- (p10);
    }
\begin{scope}[shift={(-5,0)}]
\BsigmaT
\fill[pattern=north west lines, pattern color=gray, opacity=0.8] (p8) -- (p9) -- (p10) -- cycle;
\draw (p8) node[anchor=north] {{\footnotesize codim-$0$ faces of $\sigma$}};
\end{scope}
\begin{scope}[shift={(0,0)}]
\BsigmaT
\fill[pattern=north west lines, pattern color=gray, opacity=0.8] (p8) -- (p9) -- (p10) -- cycle;
\draw[line width=4pt, cbgreen, cap=round] (p1) -- (p2);
\draw[ultra thick, white, cap=round] (p1) -- (p2);
\draw[line width=4pt, cbblue, cap=round] (p2) -- (p3);
\draw[ultra thick, white, cap=round] (p2) -- (p3);
\draw[line width=4pt, cbred, cap=round] (p1) -- (p3);
\draw[ultra thick, white, cap=round] (p1) -- (p3);
\draw[very thick, gray, fill=gray] (p1) circle[radius=2.25pt];
\draw[very thick, gray, fill=gray] (p2) circle[radius=2.25pt];
\draw[very thick, gray, fill=gray] (p3) circle[radius=2.25pt];
\draw[very thick, cbgreen, fill=cbgreen] (p6a) circle[radius=2.25pt];
\draw[very thick, cbblue, fill=cbblue] (p8) circle[radius=2.25pt];
\draw[ultra thick, cbred] (p4) -- (p5);
\draw[very thick, cbred, fill=cbred] (p4) circle[radius=2.25pt];
\draw[very thick, cbred, fill=cbred] (p5) circle[radius=2.25pt];
\draw[thick, cbgreen, -latex, shorten >= 0.4em, shorten <= 0.4em] (p6a) to node[anchor=north] {{\footnotesize $\pi$}} ($0.5*(p1)+0.5*(p2)$);
\draw[thick, cbred, -latex, shorten >= 0.4em, shorten <= 0.4em] ($0.5*(p4)+0.5*(p5)$) to[out=-90,in=30] node[anchor=north west] {{\footnotesize $\pi$}} ($0.5*(p1)+0.5*(p3)$);
\draw[thick, cbblue, -latex, shorten >= 0.4em, shorten <= 0.4em] (p8) to node[anchor=east] {{\footnotesize $\pi$}} ($0.5*(p2)+0.5*(p3)$);
\draw (p8) node[anchor=north] {{\footnotesize codim-$1$ faces of $\sigma$}};
\end{scope}
\begin{scope}[shift={(5,0)}]
\BsigmaT
\fill[pattern=north west lines, pattern color=gray, opacity=0.8] (p8) -- (p9) -- (p10) -- cycle;
\draw[very thick, cbgreen, fill=white] (p1) circle[radius=2.25pt];
\draw[very thick, cbblue, fill=white] (p2) circle[radius=2.25pt];
\draw[very thick, cbred, fill=white] (p3) circle[radius=2.25pt];
\draw[very thick, cbgreen, fill=cbgreen] (p6b) circle[radius=2.25pt];
\draw[very thick, cbblue, fill=cbblue] (p7) circle[radius=2.25pt];
\draw[ultra thick, cbred, cap=round] (p9) -- (p10);
\draw[very thick, cbred, fill=cbred] (p9) circle[radius=2.25pt];
\draw[very thick, cbred, fill=cbred] (p10) circle[radius=2.25pt];
\draw[thick, cbgreen, -latex, shorten >= 0.4em, shorten <= 0.4em] (p6b) to node[anchor=north east] {{\footnotesize $\pi$}} (p1);
\draw[thick, cbblue, -latex, shorten >= 0.4em, shorten <= 0.4em] (p7) to node[anchor=south] {{\footnotesize $\pi$}} (p2);
\draw[thick, cbred, -latex, shorten >= 0.4em, shorten <= 0.4em] ($0.5*(p9)+0.5*(p10)$) to node[anchor=west] {{\footnotesize $\pi$}} (p3);
\draw (p8) node[anchor=north] {{\footnotesize codim-$2$ faces of $\sigma$}};
\end{scope}
\end{tikzpicture}\\
\vspace{1em}
\begin{tikzpicture}
\newcommand{\BsigmaT}{
    \coordinate (p1) at (0,1);
    \coordinate (p2) at (-1,0);
    \coordinate (p3) at (1,0);
    \coordinate (p4) at ($(p1) + (1.75,0)$);
    \coordinate (p5) at ($(p1) + (0.5,1.5)$);
    \coordinate (p6a) at ($(p1) + (-1.75,0.25)$);
    \coordinate (p6b) at ($(p1) + (-0.75,1.25)$);
    \coordinate (p7) at ($(p2) + (-1,-0.5)$);
    \coordinate (p8) at ($(p2) + (0.75,-1.5)$);
    \coordinate (p9) at ($(p3) + (0,-1)$);
    \coordinate (p10) at ($(p3) + (0.75,-1.25)$);
    \draw[gray, densely dashed, cap=round] (p3) -- (p5); \draw[line width=4pt, white, cap=round] (p1)--(p4); \draw[gray, densely dashed, cap=round] (p1) -- (p4);
    \draw[gray, densely dashed, cap=round] (p1) -- (p5);
    \draw[gray, densely dashed, cap=round] (p1) -- (p6a);
    \draw[gray, densely dashed, cap=round] (p1) -- (p6b);
    \draw[gray, densely dashed, cap=round] (p2) -- (p6a);
    \draw[gray, densely dashed, cap=round] (p2) -- (p7);
    \draw[gray, densely dashed, cap=round] (p2) -- (p8);
    \draw[gray, densely dashed, cap=round] (p3) -- (p8);
    \draw[gray, densely dashed, cap=round] (p3) -- (p4);
    \draw[gray, densely dashed, cap=round] (p3) -- (p9);
    \draw[gray, densely dashed, cap=round] (p3) -- (p10);
    \draw[gray, densely dashed, cap=round] (p4) -- (p10);
    \draw[gray, densely dashed, cap=round] (p8) -- (p9);
    \draw[ultra thick, gray, cap=round] (p1) -- (p2) -- (p3) -- cycle;
    \foreach \pt in {p1,p2,p3}{
        \draw[thick,gray,fill=gray] (\pt) circle[radius=1.75pt];
    }
    \draw ($0.33*(p1)+0.33*(p2)+0.33*(p3)$) node[text=black] {{\footnotesize $\sigma$}};
    \draw[thick, densely dashed, gray, cap=round] (p4) -- (p5) -- (p6b) -- (p6a) -- (p7) -- (p8) -- (p10) -- cycle;
    \draw[thick, densely dashed, gray, cap=round] (p8) -- (p9) -- (p10);
    \draw[thick, densely dashed, gray, cap=round] (p8) -- (p10);
    }
\begin{scope}[shift={(-5,0)}]
\BsigmaT
\fill[pattern=north west lines, pattern color=gray, opacity=0.8] (p8) -- (p9) -- (p10) -- cycle;
\draw (p5) node[text=black, anchor=south east] {{\footnotesize $S_0$}};
\end{scope}
\begin{scope}[shift={(0,0)}]
\BsigmaT
\fill[pattern=north west lines, pattern color=gray, opacity=0.8] (p8) -- (p9) -- (p10) -- cycle;
\draw[ultra thick, black, cap=round] (p4) -- (p5);
\foreach \pt in {p4,p5,p6a,p8}{
    \draw[very thick, black, fill=black] (\pt) circle[radius=2.25pt];
}
\draw (p5) node[text=black, anchor=south east] {{\footnotesize $S_1$}};
\end{scope}
\begin{scope}[shift={(5,0)}]
\BsigmaT
\fill[black,fill opacity=0.25] (p8) -- (p9) -- (p10) -- cycle;
\draw[ultra thick, black, cap=round] (p4) -- (p5) -- (p6b) -- (p6a) -- (p7) -- (p8) -- (p9) -- (p10) -- (p4);
\draw[ultra thick, black, cap=round] (p8) -- (p10);
\foreach \pt in {p4,p5,p6b,p6a,p7,p8,p9,p10}{
    \draw[very thick, black, fill=black] (\pt) circle[radius=2.25pt];
}
\draw (p5) node[text=black, anchor=south] {{\footnotesize $S_2 = \hat{S}(\sigma,T_n)$ }};
\end{scope}
\end{tikzpicture}
\captionsetup{margin=0.75in}
\captionof{figure}{The top row shows $\pi \colon V(\hat S(\sigma, T_n)) \rightarrow \mathcal{F}(\sigma)$ for faces of $\sigma$ in each codimension. The bottom row shows the filtration $S_0 \subset S_1 \subset \cdots \subset \hat{S}(\sigma,T_n)$.}\label{fig:combinatorial-filtration}
\end{figure}
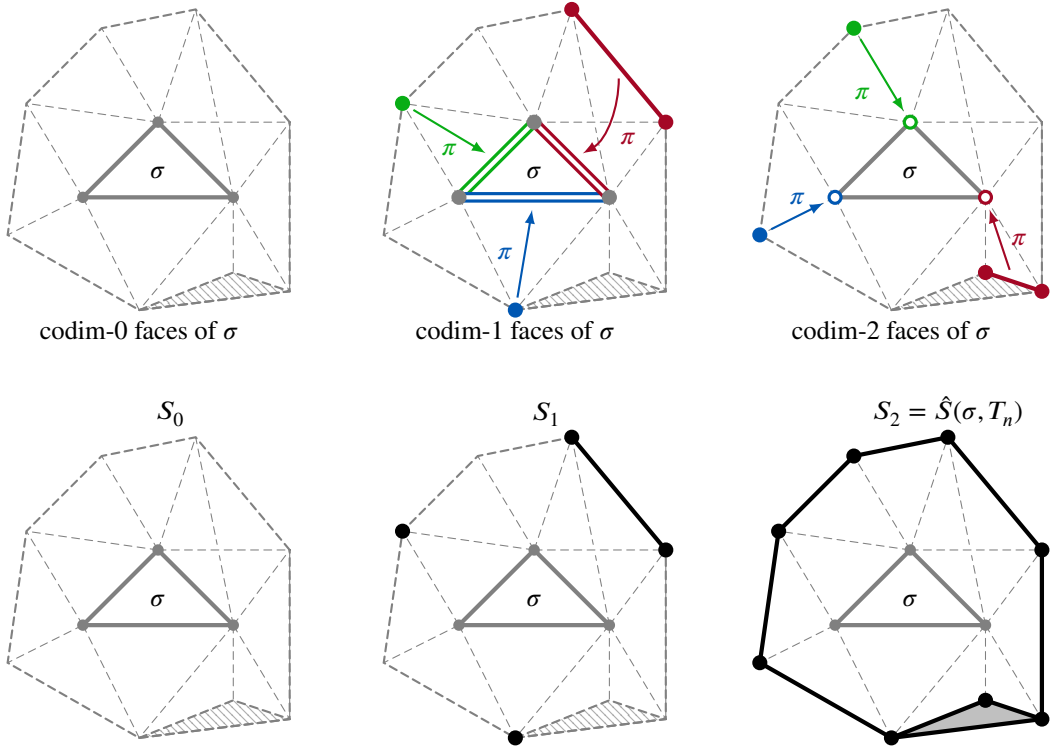

To compute the cohomological dimension of $\hat S(\sigma, T_n)$, we will analyze how the cohomological dimension changes as we work our way up this filtration. To do this, we now analyze how each simplicial complex $S_i$ in the filtration is obtained from the previous simplicial complex $S_{i-1}$. The vertices of the subcomplex  $S_{i}$ are obtained from $S_{i-1}$ by adding in all vertices $v\in \hat S(\sigma, T_n)$ with the property that $\pi(v)$ is a codimension $i$ face of $\sigma$. Take a codimension $i$ face $\tau \subset \sigma$, and consider the set of vertices $L_\tau\coloneqq\pi^{-1}(\tau) \subset S_i$. As $\tau$ ranges over all the codimension $i$ faces, the sets $L_\tau$ partition the vertices in $S_i \setminus S_{i-1}$. This gives us a description of the vertices that are added, as they are given by the disjoint union $\coprod L_\tau$ as $\tau$ ranges over the codimension $i$ faces. Next we need to describe the simplices that are added.
 
\begin{lem}\label{lem:chunks-dont-connect}
    If $\tau$, $\tau^\prime$ are distinct codimension $i$ faces of $\sigma$, then there are no edges joining vertices of $L_\tau$ to vertices of $L_{\tau^\prime}$.
\end{lem}

\begin{proof}
Consider a pair of vertices $v\in L_\tau$ and $v^\prime \in L_{\tau^\prime}$, where $\tau, \tau^\prime$ are distinct codimension $i$ faces. Since $\tau \neq \tau^\prime$ have the same codimension in $\sigma$, there is a vertex $w\in \tau\setminus \tau^\prime$, and a vertex $w^\prime\in \tau^\prime\setminus \tau$. The vertices $v\in L_\tau$ and $w\in \tau$ are connected by an edge by the defining property of $L_\tau$. Similarly $v^\prime$ and $w^\prime$ are connected by an edge. Then $w,w^\prime$ are both vertices in the simplex $\sigma$, so they are connected by an edge. For a contradiction, assume that $v, v^\prime$ are connected by an edge. Then the ordered set of vertices $v,w,w^\prime,v^\prime$ defines a $4$-cycle in $T_n$, which we know is $5$-large, hence the $4$-cycle must contain a diagonal. If $v$ is connected by an edge to $w^\prime$, we obtain a contradiction to the fact that $\pi(v)=\tau$ and $w^\prime\notin \tau$. A similar contradiction occurs if $v^\prime$ is connected by an edge to $w$. We conclude that $v$, $v^\prime$ cannot be connected by an edge, establishing the lemma.
\end{proof}

As a consequence of \Cref{lem:chunks-dont-connect}, any of the new simplices that are added to $S_{i-1}$ either span vertices in a single $L_\tau$ or connect some vertices in $L_\tau$ to vertices in $S_{i-1}$. Denote the collection of vertices in $\lk(\tau, T_n)\cap S_{i-1}$ by $\partial L_\tau$.
 
\begin{lem}\label{lem:chunks-come-from-links}
    The subcomplex of $S_i$ spanned by the vertices $L_\tau \cup \partial L_\tau$ is naturally isomorphic to the simplicial complex $\lk(\tau, T_n) \setminus \hat \sigma$, where $\hat \sigma \subset \lk(\tau, T_n)$ is the simplex in $\lk(\tau, T_n)$ corresponding to $\sigma$. Under this isomorphism, the set of vertices $\partial L_\tau$ corresponds to the vertices in $\lk(\tau, T_n)$ at distance one from $\hat \sigma$. Moreover, this subcomplex contains all the simplices in $S_i$ that contain a vertex of $L_\tau$.
\end{lem}
 
\begin{proof}
By construction, the vertices in $L_\tau$ can be identified with vertices in $\lk(\tau, T_n)$ by viewing the link as a subcomplex in the flag complex $T_n$. We claim that the vertices of ${\lk(\tau, T_n) \setminus \{L_\tau \cup \hat \sigma\}}$ are precisely the vertices of $\lk(\tau, T_n)$ at combinatorial distance one from $\hat \sigma$. Indeed, if $v$ is at combinatorial distance one from $\hat \sigma$, this means there is some vertex $w\in \hat \sigma$ connected by an edge to $v$. In terms of the embedding of $\lk(\tau, T_n)$ into $T_n$, the simplex $\hat \sigma \subset \lk(\tau, T_n)$ corresponds to the face of $\sigma$ opposite $\tau$, so $w\in \sigma \setminus \tau$. Since $v\in \lk(\tau, T_n)$ is connected by an edge to $w$, we have that $\tau \cup \{w\} \subseteq \pi(v)$. This forces the codimension of $\pi(v)$ to be at most $i-1$, which means that $v\in S_{i-1}$. This argument can be reversed to show that the vertices in $L_\tau$ are precisely the vertices in $\lk(\tau, T_n)$ whose combinatorial distance to $\hat \sigma$ is at least two. Of course, the vertices of $\hat \sigma$ in $\lk(\tau, T_n)$ are precisely those in the face of $\sigma$ opposite $\tau$, so they do not lie in $S_{i-1}$. We conclude that the set $\partial L_\tau$ consists of precisely the vertices in $\lk(\tau, T_n)$ that are exactly at distance $1$ from $\hat \sigma$.

Thus all the vertices in $L_\tau \cup \partial L_\tau$ lie in $\lk(\tau, T_n)$, and correspond to those vertices at combinatorial distance at least one from $\hat \sigma \subset \lk(\tau, T_n)$. Since $T_n$ is flag, the subcomplex they span is naturally isomorphic to the subcomplex of $\lk(\tau, T_n)$ spanned by all its vertices that are not in $\hat \sigma$.

To check the last statement, we note that \Cref{lem:chunks-dont-connect} implies that if $w\in S_i$ is a vertex connected to some vertex $v\in L_\tau$, then $w\in L_\tau \cup S_{i-1}$. It then suffices to check that any such $w\in S_{i-1}$ in fact lies in $\partial L_\tau = S_{i-1} \cap \lk(\tau, T_n)$. In other words, it is sufficient to check that such a $w$ lies in $\lk(\tau, T_n)$. Since $w\in S_{i-1}$, we know that $\pi(w)$ is a face of codimension $\leq i-1$. Since $\tau$ has codimension $i$, there exists a vertex $w^\prime\in \sigma\setminus \tau$ connected to $w$. We now claim that $\tau \subset \pi(w)$, and hence that $w \in \lk(\tau, T_n)$. Indeed, let $v^\prime \in \tau$ be an arbitrary vertex, and observe that the ordered $4$-tuple of vertices $v, w, w^\prime, v^\prime$ forms a $4$-cycle. Since $v$ is not connected to $w^\prime$, since $w^\prime\notin \tau=\pi(v)$, $5$-largeness of $T_n$ forces the $w$ to be connected to $v^\prime$. We conclude that any vertex $w\in S_{i-1}$ connected to a vertex of $L_\tau$ is also in $\lk(\tau, T_n)$, and hence lies in the subset $\partial L_\tau = S_{i-1} \cap \lk(\tau, T_n)$. This checks the last statement, and completes the proof of the lemma.
\end{proof}
 
Finally, we will require the following technical lemma.

\begin{lem}\label{lem:cdim-of-gluings}
    Consider a pair of simplicial complexes $Z, Z^\prime$ that can be expressed as gluings $Z = A \cup_C B$ and $Z^\prime = A \cup _C B^\prime$, where each inclusion of $C$ into $A$ induces the same maps on reduced cohomology. Assume $B$ is contractible, $\cdim(Z)\leq n-1$, and $\cdim(B^\prime)\leq n-1$. Then $\cdim (Z^\prime) \leq n-1$.
\end{lem}

\begin{proof}
We first consider the Mayer--Vietoris sequence associated with the gluing $Z=A\cup_C B$. Since $B$ is contractible by hypothesis, the sequence reduces to
    \[ \cdots \longrightarrow \overline{H}{}^i(Z) \longrightarrow \overline{H}{}^i(A) \longrightarrow \overline{H}{}^i(C)\longrightarrow \overline{H}{}^{i+1}(Z) \longrightarrow \cdots \]
We also know by hypothesis that $\cdim(Z)\leq n-1$, so $\overline{H}{}^i(Z)=0$ for $i\geq n$. The long exact sequence now tells us that the map $\overline{H}{}^i(A) \rightarrow \overline{H}{}^i(C)$ induced by the inclusion $C\hookrightarrow A$ is an isomorphism when $i\geq n$ and a surjection when $i=n-1$.

Now we turn our attention to $Z^\prime= A\cup_C B^\prime$, obtaining the Mayer--Vietoris sequence:
    \[ \cdots \longrightarrow \overline{H}{}^{i-1}(C) \longrightarrow \overline{H}{}^i(Z^\prime) \longrightarrow \overline{H}{}^i(A)\oplus \overline{H}{}^i(B^\prime) \longrightarrow \overline{H}{}^i(C)\longrightarrow \cdots \]
By hypothesis, we have that $\cdim(B^\prime) \leq n-1$, so $\overline{H}{}^i(B^\prime)=0$ for $i\geq n$. In particular, when $i\geq n$ the sequence above reduces to:
    \[ \cdots \longrightarrow \overline{H}{}^{i-1}(C) \longrightarrow \overline{H}{}^i(Z^\prime) \longrightarrow \overline{H}{}^i(A) \longrightarrow \overline{H}{}^i(C)\longrightarrow \cdots \]
From our earlier analysis, we also know that the map $\overline{H}{}^i(A) \rightarrow \overline{H}{}^i(C)$ is an isomorphism when $i\geq n$, and a surjection when $i=n-1$. This forces $\overline{H}{}^i(Z^\prime)=0$ in the range $i\geq n$, which proves the lemma.
\end{proof}

With our preliminaries in place, we can now establish \Cref{prop:cd-spherical-nbhds}, and hence complete the proof of the Main Theorem.

\begin{proof}[Proof of \Cref{prop:cd-spherical-nbhds}]
We want to show that $\cd(S(\sigma, T_n))\leq n-1$ for every nonempty simplex $\sigma \subset T_n$. From \Cref{lem:metric-to-combinatorial-sphere}, we have a homotopy equivalence $S(\sigma, T_n) \simeq \hat S(\sigma, T_n)$, so we instead show that $\cd(\hat S(\sigma, T_n))\leq n-1$. As discussed above, the set $\hat S(\sigma, T_n)$ has a filtration
    \[ S_0 \subset S_1 \subset \cdots \subset S_k = \hat S(\sigma, T_n). \]
We will consider the cohomological dimension of each stage in this filtration, and check that it always remains less than or equal to $n-1$.

At the bottom of the filtration we have the space $S_0$. This is the full subcomplex of $\hat S(\sigma, T_n)$ containing all vertices $v$ that satisfy $\pi(v)=\sigma$. This means that the vertex $v$ is adjacent to every vertex of $\sigma$. Since $T_n$ is a flag complex, we see that the join $v*\sigma$ is a $(k+1)$-simplex containing $\sigma$, and hence $v$ can be identified with a vertex in $\lk(\sigma, T_n)$. It follows that the subcomplex $S_0$ is precisely the copy of $\lk(\sigma, T_n)$ inside the complex $\hat S(\sigma, T_n)$. Applying \Cref{lem:cdim-of-links} we see that $\cd(S_0)= \cd(\lk(\sigma, T_n))\leq n-1$. This establishes the base case of the induction. Next we consider how the cohomological dimension changes when we go from some $S_{i-1}$ to $S_{i}$. 

By induction, we may now assume that $\cdim(S_{i-1})\leq n-1$. From \Cref{lem:chunks-dont-connect} and \Cref{lem:chunks-come-from-links}, we see that $S_{i}$ is obtained by gluing on a finite collection of subcomplexes to $S_{i-1}$. This collection is parametrized by the $\binom{k+1}{i}$ codimension $i$ faces of $\sigma$. Let us now analyze how the topology changes each time we add one of these subcomplexes. We let $A_1, A_2, \ldots $ denote the collection of spaces we are adding on to $S_{i-1}$ and $B_j^\prime$ denote the result of adding $A_1, \ldots , A_j$ to $S_{i-1}$. Note that once we have reached $j={\binom{k+1}{i}}$, we have added on the subsets corresponding to each codimension $i$ face, hence $B_j^\prime=S_{i+1}$. Thus it will be sufficient to argue that each $\cdim(B_j^\prime) \leq n-1$, which we now do by a second inductive argument.

We will use \Cref{lem:cdim-of-gluings} to see that the cohomological dimension cannot increase as we go from $B_j^\prime$ to $B_{j+1}^\prime$. In the notation of \Cref{lem:cdim-of-gluings}, we are setting:
    \begin{itemize}
        \item $Z\coloneqq\lk(\tau, T_n)$ where $\tau$ is one of the codimension $i$ faces of $\sigma$, so $\cdim(Z)\leq n-1$ by \Cref{lem:cdim-of-links};
        \item $A\coloneqq A_{j+1}$ is the subcomplex of $Z$ spanned by all vertices that are not in the subcomplex $\hat \sigma$ from \Cref{lem:chunks-come-from-links} applied to $\tau$;
        \item $B$ is the subcomplex of $Z$ consisting of all simplices incident to $\hat \sigma$ (from \Cref{lem:chunks-come-from-links}), which coincides with $\hat B(\hat \sigma, Z)$, and is thus contractible by \Cref{lem:metric-to-combinatorial-sphere}, since $T_n$, hence $Z$, is flag and $5$-large;
        \item $B^\prime\coloneqq B_j^\prime$ is the union of $S_{i-1}$ with finitely many of the subcomplexes $A_1, \ldots , A_j$ already added on, so by induction we can assume that $\cdim(B^\prime) \leq n-1$;
        \item $C \coloneqq\partial L_\tau \subseteq S_{i-1}$ from the conclusions about $\partial L_\tau$ in \Cref{lem:chunks-come-from-links}, hence $C$ is contained in $B$ and $B^\prime$.
    \end{itemize}
Thus all the hypotheses of \Cref{lem:cdim-of-gluings} are satisfied, which allows us to conclude that the space $Z^\prime=A\cup_C B^\prime$ satisfies $\cdim(Z^\prime)\leq n-1$. However, observe that $Z^\prime$ is exactly the result of adding on $A=A_{j+1}$ to $B^\prime=B_j^\prime$, hence it is the space $B_{j+1}^\prime$, thus completing the inductive step. After performing this step for the finitely many codimension $i$ faces, we see that $\cdim(S_i)\leq n-1$. This completes the inductive step, and continuing this procedure up the filtration completes the proof that $\cd(S(\sigma, T_n))=\cd(\hat S(\sigma, T_n)) \leq n-1$.
\end{proof}

\subsection*{Declarations of Competing Interests} The authors have no competing interests to declare that are relevant to the content of this article.  

\subsection*{Data Availability Statement} Data sharing is not applicable to this article as no datasets were generated or analyzed during the current study.

\bibliographystyle{plain}
\bibliography{references.bib}

\begin{thebibliography}{10}

\bibitem{BM}
Ludovico Battista and Bruno Martelli.
\newblock Hyperbolic 4-manifolds with perfect circle-valued {M}orse functions.
\newblock {\em Trans. Amer. Math. Soc.}, 375:2597--2625, 2022.

\bibitem{BB}
Mladen Bestvina and Noel Brady.
\newblock Morse theory and finiteness properties of groups.
\newblock {\em Invent. Math.}, 129(3):445--470, 1997.

\bibitem{CharneyFarber}
Ruth Charney and Michael Farber.
\newblock Random groups arising as graph products.
\newblock {\em Algebr. Geom. Topol.}, 12(2):979--995, 2012.

\bibitem{Davis}
Michael~W. Davis.
\newblock {\em The geometry and topology of {C}oxeter groups}, volume~32 of
  {\em London Mathematical Society Monographs Series}.
\newblock Princeton University Press, Princeton, NJ, 2008.

\bibitem{DavisKahle}
Michael~W. Davis and Matthew Kahle.
\newblock Random graph products of finite groups are rational duality groups.
\newblock {\em J. Topol.}, 7(2):589--606, 2014.

\bibitem{F}
Sam~P. Fisher.
\newblock Algebraic fibring of a hyperbolic 7-manifold.
\newblock {\em Bull. Lond. Math. Soc.}, 55(3):1347--1357, 2023.

\bibitem{HaglundWise}
Fr\'{e}d\'{e}ric Haglund and Daniel~T. Wise.
\newblock Special cube complexes.
\newblock {\em Geom. Funct. Anal.}, 17(5):1551--1620, 2008.

\bibitem{IMM1}
Giovanni Italiano, Bruno Martelli, and Matteo Migliorini.
\newblock Hyperbolic manifolds that fiber algebraically up to dimension 8.
\newblock {\em J. Inst. Math. Jussieu}, pages 1--38, 2022.

\bibitem{IMM2}
Giovanni Italiano, Bruno Martelli, and Matteo Migliorini.
\newblock Hyperbolic 5-manifolds that fiber over $\mathbb{S}^1$.
\newblock {\em Invent. Math.}, 231:1--38, 2023.

\bibitem{JNW}
Kasia Jankiewicz, Sergey Norin, and Daniel~T. Wise.
\newblock Virtually fibering right-angled {C}oxeter groups.
\newblock {\em J. Inst. Math. Jussieu}, 20(3):957--987, 2021.

\bibitem{Lacher}
R.~C. Lacher.
\newblock Cell-like mappings. {I}.
\newblock {\em Pacific J. Math.}, 30:717--731, 1969.

\bibitem{osajda}
Damian Osajda.
\newblock A construction of hyperbolic {C}oxeter groups.
\newblock {\em Comment. Math. Helv.}, 88(2):353--367, 2013.

\bibitem{Sageev}
Michah Sageev.
\newblock {$\rm CAT(0)$} cube complexes and groups.
\newblock In {\em Geometric group theory}, volume~21 of {\em IAS/Park City
  Math. Ser.}, pages 7--54. Amer. Math. Soc., Providence, RI, 2014.

\bibitem{ScheslerZaremsky}
Eduard Schesler and Matthew C.~B. Zaremsky.
\newblock Random subcomplexes of finite buildings, and fibering of commutator
  subgroups of right-angled {C}oxeter groups.
\newblock {\em J. Topol.}, 16(1):20--56, 2023.

\end{thebibliography}

\end{document}